\documentclass[12pt]{article}
\usepackage{color}
\usepackage{tikz}
\usepackage{hyperref}
\usepackage{enumerate}
\usepackage{amssymb}
\usepackage{amsmath}
\usepackage{color}
\usepackage{amsthm}

\newtheorem{thm}{Theorem}[section]
\newtheorem{lem}{Lemma}[section]
\newtheorem{cor}[thm]{Corollary}
\newtheorem{prop}[lem]{Proposition}
\newtheorem{rem}[lem]{Remark}
\newtheorem{ob}[lem]{Observation}

\newtheorem{defn}[lem]{Definition}
\def\R{\mathbb{R}}
\def\C{\mathbb{C}}
\def\Im{\mathrm{Im}}

\def\pat{\partial}
\def\patt{\tilde{\partial}}
\def\S{\operatorname{S}}
\def\H{\mathcal{H}}
\def\T{\operatorname{T}}
\def\Dom{\mathrm{Dom}}
\def\gt{\tilde{g}_T}

\def\td{\tilde{d}_T}
\def\Crit{\mathrm{Crit}}
\def\tPhi{\tilde{\Phi}^t}

\begin{document}
\def\o{\omega}
\def\O{\Omega}

\hypersetup{hypertex=true,
            colorlinks=true,
            linkcolor=blue,
            anchorcolor=blue,
            citecolor=blue}
\title{Witten deformation for noncompact manifolds with bounded geometry}

\author{Xianzhe Dai\thanks{Department of Mathematics, UCSB, Santa Barbara CA 93106, dai@math.ucsb.edu. \ \ \ \  \ \  \
		Partially supported by the Simons Foundation}
	\and Junrong Yan\footnote{Department of Mathematics,  UCSB, Santa Barbara CA 93106, j\_yan@math.ucsb.edu}
}

\maketitle

\abstract{ Motivated by the Landau-Ginzburg model,
we study the Witten deformation on a noncompact manifold with bounded geometry, together with some tameness condition on the  growth of the Morse function $f$ near infinity.
We prove that the cohomology of the Witten
deformation $d_{Tf}$ acting on the complex of smooth $L^2$ forms is isomorphic to the cohomology of Thom-Smale complex of $f$ as well as the relative cohomology of a certain pair $(M, U)$ for sufficiently large $T$.  We establish an Agmon estimate for eigenforms of the Witten Laplacian which plays an essential role in identifying these cohomologies via 
 Witten's instanton complex, defined in terms of eigenspaces of the Witten Laplacian for small eigenvalues. As an application we obtain the strong Morse inequalities in this setting.
}

\section{Introduction}

\subsection{Overview}
In the extremely influential paper \cite{witten1982supersymmetry}, Witten introduced a deformation of the de Rham complex by considering the new differential $d_f=d+df,$ where $d$ is the usual exterior derivative on forms, and $f$ is a Morse function. Setting  $$d_{Tf}:=d+Tdf, $$ Witten observed that when $T>0$ is large enough, the eigenfunctions of the small eigenvalues for the
corresponding deformed Hodge-Laplacian, the so called Witten Laplacian,  concentrate at the critical
points of $f$. As a result,  Witten deformation builds a direct bridge between the Betti numbers and the Morse indices of the critical points of $f$.

Witten deformation on closed manifolds has produced a whole range of beautiful applications, from Demailly's holomorphic Morse inequalities\cite{demailly1985holmorse}, to the proof of Ray-Singer conjecture and its generalization by Bismut-Zhang \cite{bismutzhang1992cm}, to the instigation of the development of Floer homology theory.

Although the Witten deformation on noncompact manifolds are much less studied and understood, there are previous interesting work in the direction. In \cite{dimcasaito93polymap} the cohomology of an affine algebraic variety is related to that of the Witten complex of $\mathbb C^m$, see also \cite{farbershustin00polyform} for further development.

This paper is motivated by the study of Landau-Ginzburg models (c.f.\cite{hori2003mirror}), which, according to Witten \cite{witten1993phases},  are simply different phase of Calabi-Yau manifolds, and hence equivalent to Calabi-Yau manifolds. Suppose there is a non-trivial holomorphic function $W$ (the superpotential) on a noncompact Kahler manifold $M^n$ ($n=\dim _{\mathbb C} M$), then one considers the Witten-deformation of $\partial$ operator:
\[\bar\partial_W=\bar\partial-\frac{i}{2}\partial W,\]
as its cohomology describes the quantum ground states of the Landau-Ginzburg model $(M, W)$.
If $W$ is also a Morse function with $k$ critical points, then complex Morse theoretic consideration leads to the expectation that
\[H^l_{\bar\partial_W}(M)=\begin{cases}
\C^k, \mbox{ if } l=n\\
0, \mbox{ otherwise}.
\end{cases}\]
For the mathematical study of LG models and their significant applications we point out the important work \cite{fanjarvisruan2013quantum}.

In this paper, we consider the more general case for Riemannian manifolds: we explore the relations between the Thom-Smale complex for a Morse function $f$ on a noncompact manifold $M$ and the deformed de Rham complex with respect to $f$.  The first difficulty one encounters here is the presence of continuous spectrum on a noncompact manifolds and for that one has to impose certain tameness conditions. This consists of the bounded geometry requirement for the manifold as well as growth conditions for the function. The notion of strong tameness is introduced in \cite{fan2011schr} in the K\"ahler setting which guarantees the discreteness of the spectrum for the Witten Laplacian. Here we introduce a slightly weaker notion which allows continuous spectrum but only outside a large interval starting from $0$.

It is important to note that, and this is another new phenomenon in the noncompact case,   the Thom-Smale complex may not be  a complex in general. Namely, the square of its boundary operator need not be zero, since $M$ is noncompact. However we prove that with the tameness condition, it is. 

The crucial technical  part of our work is the Agmon estimate for eigenforms of the Witten Laplacian which is essential in extending the usual analysis from compact setting to the noncompact case.  The Agmon estimate was discovered by S. Agmon in his study of $N$-body
Schr\"odinger operators in the Euclidean setting and has found many important applications. The exponential decay of the eigenfunction is expressed in terms of the so-called Agmon distance, Cf. \cite{agmon2014lectures}. We make essential use of this Agmon estimate to carry out the isomorphism between the Witten instanton complex defined in terms of eigenspaces corresponding to the small eigenvalues with the Thom-Smale complex defined in terms of the critical point data of the function.  We remark that the Agmon estimate near the critical points  also plays important role in the compact case , see,  e.g. \cite{bismutzhang1992cm}. The novelty here is that we make essential use of  the exponential decay at spatial infinity provided by the Agmon estimate.

As an application of our results on noncompact manifolds, we deduce corresponding results for manifolds with boundaries which generalize recent work of \cite{laudenbach2011morse}, \cite{lu2017thomsmale}.

Finally we would also like to point out the preprint \cite{fanfang2016torsion} which has provided further motivation and inspiration for us. 

In the rest of the introduction we give precise statements of our main results after setting up our notations. In subsequent work we will develop the local index theory and the Ray-Singer torsion for the Witten deformation in the noncompact setting.
\newline

{\em Acknowledgment:  We thank Shu Shen for interesting discussions and for providing us with an example of Thom-Smale complex not being a complex.
}

\subsection{Notations and basic setup}
\numberwithin{equation}{section}
Let $(M,g)$ be a noncompact connected complete Riemannian manifold with metric
$g$. $(M,g)$ is said to have bounded geometry, if the following
conditions hold:
\begin{enumerate}
\item the injectivity radius $r_0$ of $M$ is positive.
\item $|\nabla^m R|\le C_m$, where $\nabla^m R$ is the $m$-th
covariant derivative of the curvature tensor and $C_m$ is a constant
only depending on $m$.
\end{enumerate}

On such a manifold, the Sobolev constant is uniformly bounded, see e.g. \cite{cgt}.
Now let $f:M\mapsto\R$ be a smooth function. In \cite{fan2011schr},  the notion of strong tameness for the triple
$(M,g,f)$ is introduced.

\begin{defn} The triple $(M,g,f)$ is said to be strongly tame, if $(M, g)$ has bounded geometry and
\begin{equation*}
\lim\sup_{p\to\infty}\frac{|\nabla^2 f|(p)}{|\nabla f|^2(p)}=0,
\end{equation*}
and
\[\lim_{p\to\infty}|\nabla f|\to\infty,\]
where $\nabla f, \nabla^2 f$ are the gradient and Hessian of $f$ respectively.
\end{defn}

\begin{rem}\label{rem}
	
	Fix $p_0\in M$, and let $d$ be the distance function induced by $g$.   Here $p\to\infty$ simply means that $d(p,p_0)\to \infty.$
\end{rem}

In this paper we only need the following weaker condition.
\begin{defn} The triple $(M,g,f)$ is said to be well tame, if $(M, g)$ has bounded geometry and
\[
c_f:=\lim\sup_{p\to\infty}\frac{|\nabla^2 f|(p)}{|\nabla f|^2(p)}<\infty,
\]
and
\[
\epsilon_f:=\lim\inf_{p\to\infty}|\nabla f|>0. \]

\end{defn}

\def\Dom{\mathrm{Dom}}
As usual,  the metric $g$ induced a canonical metric (still denote it by $g$) on $\Lambda^*(M)$, which then defines an inner product $(\cdot,\cdot)_{L^2}$ on $\Omega^*_c(M)$:
\[(\phi,\psi)_{L^2}=\int_M (\phi, \psi)_gdvol,\phi,\psi\in \Omega^*_c(M). \]

Let $L^2\Lambda^*(M)$ be the completion of $\Omega_c^*(M)$ with respect to  $\|\cdot\|_{L^2}$, and for simplicity, we denote $L^2(M):=L^2\Lambda^0(M).$

For any $T\geq 0$, let $d_{Tf}:=d+Tdf\wedge:\Omega^*(M)\mapsto \Omega^{*+1}(M)$ be the so-called Witten deformation of de Rham operator $d$.  It is an unbounded operator on $L^2\Lambda^*(M)$ with domain $\Omega^*_c(M)$. Also, $d_{Tf}$ has a formal adjoint operator $\delta_{Tf}$, with $\Dom(\delta_{Tf})=\Omega^*_c(M),$ such that
\[(d_{Tf}\phi,\psi)_{L^2}=(\phi,\delta_{Tf}\psi)_{L^2},\phi,\psi\in\Omega^*_c(M).\]

Set $\Delta_{H,Tf}=(d_{Tf}+\delta_{Tf})^2,$ and we denote the Friedrichs extension of $\Delta_{H,Tf}$ by $\Box_{Tf}$. As we will see (Theorem \ref{esa}), if $(M, g, f)$ is well tame,  then $\Delta_{H,Tf}$ is essentially self-adjoint (and hence  $\Box_{Tf}$ is the unique self-adjoint extension).
In Section \ref{l2}, we will prove the Hodge-Kodaira decomposition when $(M, g, f)$ is well tame and $T$  large enough,
\begin{equation}\label{dec}L^2\Lambda^*(M)=\ker\Box_{Tf}\oplus\Im \bar{d}_{Tf}\oplus\Im \bar\delta_{Tf},\end{equation}
where $\bar{d}_{Tf}$ and $\bar\delta_{Tf}$ are the graph extensions of $d_{Tf}$ and $\delta_{Tf}$ respectively.

Setting  $\Omega_{(2)}^*(M):=L^2\Lambda^*(M)\cap\Omega^*(M),$
we have a chain complex (of unbounded operators)
\[\cdots \xrightarrow{d_{Tf}}\Omega_{(2)}^*(M)\xrightarrow{d_{Tf}}\Omega_{(2)}^{*+1}(M)\xrightarrow{d_{Tf}}\cdots.\]
Let $H^*_{(2)}(M,d_{Tf})$ denote the cohomology of this complex. In Section \ref{l2}, we will show that $H^*_{(2)}(M,d_{Tf})\cong\ker\Box_{Tf}$, provided $(M, g, f)$ is well tame and $T$ is large enough.

Finally we note the following well known (Cf. \cite{witten1982supersymmetry, zhang2001lectures})

\begin{prop}\label{estihdg}
 The Hodge Laplacian  
$\Delta_{H,Tf}$ has the
following local expression:
\begin{equation}
\Delta_{H,Tf}=\Delta-T \nabla^2_{e_i,e_j}f[e^i\wedge,\iota_{e_j}]+T^2|\nabla f|^2.
\end{equation}
Here $\{e_i\}$ is a local frame on $TM$ and $\{e^i\}$ is the dual frame on $T^*M.$
\end{prop}
\subsection{Main results}
In this subsection, we assume that $(M,g)$ has bounded geometry, $f$ is a Morse function with finite many critical points. 
Clearly this will be the case if $(M,g,f)$ is well tame and $f$ is Morse.

As we mentioned the main technical result here is the Agmon estimate for the eigenforms of the Witten Laplacian.

\begin{thm}\label{est}
Let $(M,g,f)$ be well tame, and $\o\in\Dom(\Box_{Tf})$ be an eigenform of $\Box_{Tf}$ whose eigenvalue is uniformly bounded in $T$. Then
\[|\o(p)|\leq CT^{(n+2)/2}\exp(-a\rho_T(p))\|\omega\|_{L^2},\]
for any $a\in(0,1)$ (provided $T$ is sufficiently large and $C$ is a constant depending on the dimension $n$, the function $f$, the curvature bound, and $a$; for the precise choice of $T, C$ see the end of Section 3). Here the definition of the Agmon distance $\rho_T(p)$ will be given in Section \ref{sec}. 
\end{thm}

The proof of the Agmon estimate, given in Section \ref{tae}, is to carry out the idea of \cite{agmon2014lectures} in this more general setting.

Set $b_i(T)=\dim H_{(2)}^i(M,d_{Tf})$. If $x$ is a critical point of $f$, denote $n_f(x)$ the Morse index of $f$ at $x.$ Let $m_i$ be the number of critical points of $f$ with Morse index $i.$ Then the strong Morse inequalities hold.
 
\begin{thm}\label{morse}
If $(M,g,f)$ is well tame, then we have the following strong Morse inequality
\[ (-1)^k \sum_{i=0}^k(-1)^i b_i(T)\leq (-1)^k \sum_{i=0}^k(-1)^im_i, \ \ \forall k\leq n,\]
provided $T$ is large enough. And the equality holds for $k=n$.
\end{thm}

In general,  $b_i(T)$ may be very sensitive to $T$.  However we have the following result regarding the indepedence of $b_i(T)$ in $T$. Assume that the Morse function $f$ satisfies the Smale transversality condition. Let $(C^*(W^u),\patt')$ be the Thom-Smale complex given by $f$. It is important to note  that in general, since $M$ is noncompact,  it could happen that $(\patt')^2\not=0$. 
Also let $c>0$ be big enough, $U_c=\{p\in M:f(p)<-c\}$ and $(\Omega^*(M,U_c),d)$ be the relative de Rham complex.

\begin{thm}\label{wts}
 If $(M,g,f)$ is well tame, then $(\patt')^2=0$, and therefore the cohomology $H^*(C^\bullet(W^u),\patt')$ is well defined. Moreover, there exists $T_0\geq 0$ (When $(M,g,f)$ is strongly tame, $T_0=0$), such that
$H^*_{(2)}(M,d_{Tf})$ is isomorphic to $H^*(C^\bullet(W^u),\patt')$ for all $T> T_0$. In addition, $H^*(C^\bullet(W^u),\patt')$, hence $H^*_{(2)}(M,d_{Tf})$ is isomorphic to the relative de Rham cohomology $H_{dR}^*(M,U_c)$.

\end{thm}

By Theorem \ref{wts}, we can refine our result of Theorem \ref{morse}
\begin{cor}
If $(M,g,f)$ is well tame, then $b_i(T)$ is independent of $T$ when $T$ is big enough. When $(M,g,f)$ is strongly tame,
$b_i(T)$ is independent of $T> 0.$
\end{cor}
As an another application of Theorem \ref{wts}, we study the Morse cohomology for compact manifolds with boundary.

Let $M$ be a compact, oriented manifold of dimension $n$ with boundary $\partial M$. Let $N_i(i\in \Lambda)$ be the connected components of $\partial M$. We fix a collar neighborhood $(0,1]\times N_i\subset M$, and let $r$ be the standard coordinate on the $(0,1]$ factor.
\begin{defn}
A smooth function $f$ on $M$ is called a transversal Morse function if it satisfies the following conditions:
\begin{enumerate}
\item $f| _{M\setminus\partial M}$ is a Morse function on the manifold $M\setminus\partial M $;
\item $f| _{\partial M}$ is a Morse function on the manifold $\partial M$.
\item For any point $x$ on the collar neighborhood, $\left.-\frac{\partial f}{\partial r}\right|_{x}\neq 0$.
\end{enumerate}
\end{defn}
For a transversal Morse function $f$ on $M$, since
$-\frac{\pat f}{\pat r}$ is continuous on any connected components of $\partial M$, so we can call $N_i$ to be positive (with
respect to $f$) if
$-\frac{\pat f}{\pat r}|_{N_i}>0$, and negative if $-\frac{\pat f}{\pat r}|_{N_i}<0.$

Let $N^+$ be the union of all positive boundaries, $N^-$ be the union of all negative boundaries. Suppose we have a partition of positive boundaries
$N^+=N^+_1\sqcup N^+_2$, and a partition of negative boundaries $N^-=N^-_1\sqcup N^-_2.$
Now We denote
\begin{itemize}
\item by $Crit^{\circ,k}(f)$ the set of internal critical points with Morse index $k$, $m_k=|Crit^{\circ,k}(f)|$;
\item by $Crit^{+,k}_{N^+_j}(f)(j=1,2)$ the set of critical points on positive boundary $N^+_j$ with Morse index $k$, $n_{k,N^+_j}=|Crit^{+,k}_{N^+_j}(f)|$;
\item by $Crit^{-,k}_{N^-_j}(f)(j=1,2)$ the set of critical points on negative boundary $N^-_j$ with Morse index $k-1$, $l_{k,N^-_j}=|Crit^{-,k}_{N^-_j}(f)|$.
\end{itemize}

Let $\Crit^{*}(f)=Cirt^{\circ,*}(f)\cup Cirt^{+,*}_{N^+_1}(f)\cup Cirt^{-,*}_{N^-_1}(f)$
\begin{thm}\label{bdy}
There is a differential $\patt':\Crit^{k}(f)\mapsto \Crit^{k+1}(f)$ making $(\Crit^{*}(f)\otimes\R,\patt')$ a chain complex. Moreover, $H^*(\Crit^\bullet(f)\otimes\R,\patt')$ is isomorphic to the relative de Rham cohomology $H^*(M,N_2^-\cup N_1^+).$

In particular, for $N_1^+=N^+,N_1^-=\emptyset,$  $H^*(\Crit^\bullet(f)\otimes\R,\patt')$ is isomorphic to the relative de Rham cohomology $H^*(M,\pat M).$
\end{thm}

\begin{cor}\label{mbdy}
Set $b_i(M,N_2^-\cup N_1^+)=\dim(H^i(M,N_2^-\cup N_1^+), $ then we have the following Morse inequalities:
\[ (-1)^k \sum_{i=0}^k(-1)^ib_i(M,N_2^-\cup N_1^+)\leq (-1)^k \sum_{i=0}^k(-1)^i(m_i+n_{i-1,N^+_1}+l_{i, N_2^-}).\]
\end{cor}
\begin{rem} 
\begin{enumerate}
\item  Theorem \ref{bdy} is a generalization of a result in \cite{laudenbach2011morse}.
\item Corollary \ref{mbdy} is a generalization of a result in \cite{lu2017thomsmale}.
\end{enumerate}
\end{rem}
\subsection{Notations and Organization}
In this article, we will generally use $\phi,\psi$ to denote differential forms, $f$  Morse function, $u,v$  functions, $\nu, \o$  eigenforms, $x,y,z$ critical points of $f$, $p,q$  general points, and $p_0$ a fixed point.

This paper is organized as follows.
In Section \ref{sowl}, we discuss the spectral theory for the Witten Laplacian in our setting. We then proceed to establish the exponential decay estimate for eigenforms of the Witten Laplacian in Section \ref{sec}. Assuming two technical results whose proofs are deferred to Section 7 (7.3) and using a lemma proved in Section 5 about the Agmon distance we prove Theorem \ref{est}, the Agmon estimate. 

In Section \ref{MorseInequality} we present the proof of the Strong Morse Inequalities, Theorem \ref{morse}, after introducing  Witten instanton complex $(F_{Tf}^{[0,1],*},d_{Tf})$. Section \ref{tsw} concerns the Thom-Smale theory in our setting. More specifically, we define the Thom-Smale complex $C^*((W^u)',\patt')$. Then we define a morphism between the Witten instanton complex and the Thom-Smale complex,  $\mathcal{J}: (F_{Tf}^{[0,1],*},d_{Tf})\mapsto C^*((W^u)',\patt').$ We prove that $\mathcal{J}$ is well defined by using the Agamon estimate, deferring the proof that $\mathcal{J}$ is a chain map to Subsection \ref{ThomSmaleComplex}. 

In Section \ref{atmwb} we give an application of our results, namely Theorem \ref{bdy}. Section \ref{tae} collects the proofs of several technical results. In the first two subsections we prove the lemmas used in the proof of Agmon estimate. In the Subsection \ref{ThomSmaleComplex}, we prove that our Thom-Smale complex is indeed a complex, i.e., $\patt'^2=0$.
The rest of the proof for Theorem \ref{wts} is in Subsection \ref{relative} and Subsection \ref{Independent}. Finally in Section \ref{l2}, which is an appendix, we discuss the Kodaira decomposition in a more general setting.

%
%
%
%
%
%
%
%

\section{The Spectrum Of Witten Laplacian}\label{sowl}

In this section we study the spectral theory of the Witten Laplacian on noncompact manifolds. In particular we establish the Kodaira decomposition and the Hodge theorem for the Witten Laplacian under our tameness condition.

\subsection{Essential self-adjointness of $d_{f}+\delta_{f}$}

\begin{thm} \label{esa} On a complete Riemannian manifold, if $$\lim\sup_{p\to\infty}\frac{|\nabla^2 f|(p)}{|\nabla f|^2(p)}<\infty,$$  then
$d_f+\delta_f$ is essentially self-adjoint.
\end{thm}
\begin{proof} 
 Since $\lim\sup_{p\to\infty}\frac{|\nabla^2 f|(p)}{|\nabla f|^2(p)}<\infty$, $\Box_{f}$ is bounded from below by Proposition \ref{estihdg}.  The rest of the proof is essentially the same as in Section 4 of \cite{chernoff1973}; see also the proof of Theorem 1.17 in \cite{gromov1983positive}.
\end{proof}

\subsection{On the spectrum of $\Box_{Tf}$}

From now on we will assume that $(M,g,f)$ is well tame. Then it follows that there exists a compact subset $K$, which can be taken to be a compact submanifold with boundaries that contains the closure of a ball of sufficiently large radius of $M$ (we will make a more specific choice of $K$ later in section \ref{tsw}), $\delta_1=\frac{1}{2}\epsilon_f,\delta_2=2c_f>0$, such that
\begin{equation} \label{conditionK} |\nabla f|>\delta_1, \  \ \  |\nabla^2 f| < \delta_2|\nabla f|^2 \  \ \mbox{on} \ M-K.\end{equation} 

Let $C_K=\max_K |\nabla^2 f|$. First, we establish the following basic lemma.
\def\dvol{\mathrm{dvol}}
\def\lt{\lambda_T}
\begin{lem}\label{t0}
	Fix any $b\in(0,1)$, there exists $T_1=T_1(c_f,  b)\geq0$ so that whenever $T\geq T_1$,  $\phi\in\Dom(\Box_{Tf})$
	\begin{eqnarray} \int_M(\Box_{Tf}\phi,\phi)\dvol & \geq & \int_M (\nabla \phi,\nabla \phi)\dvol+ \int_{M-K} b^2T^2|\nabla f|^2(\phi,\phi)\dvol \nonumber \\
	& &  - (C_R+TC_K) \int_K (\phi,\phi)\dvol . \label{bl} \end{eqnarray}
	Here $C_R$ is a constant depending only on the sectional curvature bounds of $g$.
\end{lem}   
\begin{proof} 
It suffices to show the inequality for a compactly supported smooth form.  By Proposition \ref{estihdg}, together with the Bochner-Weitzenb\"ock formula, we have
		\begin{eqnarray*}			\int_M(\Box_{Tf}\phi,\phi)\dvol & \geq & \int_M (\nabla \phi,\nabla \phi)\dvol -(C_R+TC_K) \int_K (\phi,\phi)\dvol \\
			& &  + \int_{M-K}  e_T(p) (\phi,\phi)\dvol,\end{eqnarray*}
		where $e_T=T^2|\nabla f|^2(1-\frac{4c_f}{T})$. 
		Thus, for any $b\in (0,1)$, there exists \begin{equation}\label{t1}
		T_1=T_1(c_f,  b)
		\end{equation}
		such that whenever $T\geq T_1,$ we have
\begin{eqnarray*} \int_M(\Box_{Tf}\phi,\phi)\dvol & \geq & \int_M (\nabla \phi,\nabla \phi)\dvol+  \int_{M-K}|bT\nabla f|^2(\phi,\phi)\dvol  \\  & & -(C+TC_K) \int_M (\phi,\phi)\dvol .
\end{eqnarray*}
\end{proof}
\begin{rem}
	When $(M,g,f)$ is strongly tame, we can take $T_1=0,$ but $K$ may depend on $T.$
\end{rem}

Let $\gt:=b^2T^2|\nabla f|^2g$ be a new metric on $M$ (with discrete conical singularities).
Fix $p_0 \in K,$ let $\rho_T(p)$ be the distance between $p$ and $p_0$ induced by $\gt.$ Then we have  $|\nabla \rho_T|^2=b^2T^2|\nabla f|^2$ a.e., where the gradient $\nabla$ is induced by $g$.
\begin{thm}\label{discrete}
Let $\sigma$ be the set of spectrum of $\Box_{Tf}.$ Then for any positive number $\delta < \frac{b\epsilon_f}{2}$ there is \begin{equation}\label{t2}T_2=T_2(\delta, c_f, \epsilon_f, C_R, C_K) >0\end{equation} such that  when $T\geq T_2$, $\sigma\cap[0,\delta^2 T^2]$ consists of a finite number of eigenvalues of finite multiplicity.
\end{thm}
\begin{proof}
Let $P:L^2\Lambda^*(M)\mapsto L^2\Lambda^*(M)$ be the integral of the spectral measure of $\Box_{Tf}$ on $[0,\delta^2 T^2]$. It suffices to prove that $L:=Im(P)$ is finite dimensional. For any $\phi\in L,$ we have
\begin{equation}\label{e1}\int_M(\Box_{Tf}\phi,\phi)dvol\leq \delta^2 T^2 \int_{M}|\phi|^2dvol.\end{equation}

Combining with (\ref{bl}), we have
\begin{eqnarray*} \delta^2 T^2 \int_{M}|\phi|^2dvol & \geq & \int_M (\nabla \phi,\nabla \phi)\dvol+  \int_{M-K} |bT\nabla f|^2(\phi,\phi)\dvol  \\  & & -(C_R+TC_K) \int_M (\phi,\phi)\dvol
\end{eqnarray*} provided $T\geq T_1$. That is,
\begin{eqnarray*}  \int_M (\nabla \phi,\nabla \phi)\dvol+  \int_{M-K} |bT\nabla f|^2(\phi,\phi)\dvol  & \leq &  \\   \delta^2 T^2(1 + \frac{C_R}{\delta^2T^2} +\frac{C_K}{\delta^2T}) \int_M (\phi,\phi)\dvol & &  \\
\end{eqnarray*}

Since $ |bT\nabla f|^2> (\frac{b\epsilon_f}{2})^2 T^2$ on $M-K$, $\delta < \frac{b\epsilon_f}{2}$,  there is $T_2=T_2(\delta, c_f, \epsilon_f, C_R, C_K) \geq T_1\geq 0$ such that  when $T\geq T_2$,
\begin{equation}
\int_M (\nabla \phi,\nabla \phi)\dvol \leq  \delta^2 T^2(1 + \frac{C_R}{\delta^2T^2} +\frac{C_K}{\delta^2T}) \int_K (\phi,\phi)\dvol. \label{e3}
\end{equation}

Now define $Q:L\mapsto L^2\Lambda^*(K)$, by $Qu=u|_{K}.$ By (\ref{e3}),
it's easy to see that $Q$ is injective, and $Im(Q)\subset W^{1,2}(\Lambda^*K).$ Since $W^{1,2}(\Lambda^*K)\hookrightarrow L^2\Lambda^*(K)$ is compact, $dim(L)=dim(Im(Q))$ must be finite.
\end{proof}

\begin{rem}
	Once again, if $(M,g,f)$ is strongly tame, we can take $T_2=0.$
\end{rem}

We now state the important consequence of this section. By combining Theorem \ref{esa} and Theorem \ref{discrete} with Proposition \ref{prop1}, decomposition (\ref{eq3}), we have

\begin{thm} \label{kodairahodge} 
	Assume that $(M,g,f)$ is well tame. Then we have the Kodaira decomposition
	\[L^2\Lambda^*(M)=\ker\Box_{f}\oplus\Im(\bar{d_f})\oplus\Im(\bar{\delta}_f),\]
Furthermore, the Hodge Theorem holds: 
$$H^*_{(2)}(M,d_{f})\cong\ker\Box_{f}.$$
	\end{thm}

\section{Exponential decay of eigenfunction}\label{sec}
In this section, we assume that  $(M,g,f)$ is well tame, and $T\geq T_2$, where $T_2$ is described in Lemma \ref{t0}. If $(M,g,f)$ is strongly tame, then we can just take $T_2=0.$

\def\gt{\tilde{g}_T}
Recall that $\gt:=b^2T^2|\nabla f|^2 g$, the Agmon metric on $M.$ Let $K$ be the compact set as in last section. In this and later sections we define the Agmon distance $\rho_T(p)$ be the distance between $p$ and $K$ induced by $\gt.$ Then we have $|\nabla \rho_T|^2=b^2T^2|\nabla f|^2$ a.e. $p\notin K$, where the gradient $\nabla$ is induced by $g$.

 For simplicity, denote $b^2T^2|\nabla f|^2$ by $\lambda_T$.
We need the following two technical lemmas, whose proofs are postponed to Section \ref{tae}.
\begin{lem}\label{tech1}
Assume $w\in L^2(M),0\leq u\in L^2(M)$, and $(\Delta +\lt)u\leq w$ outside the compact subset $K\subset M$ in the weak sense. That is
\[\int_{M-K}\nabla u\nabla v+\lt u v \dvol\leq\int_{M-K}w\cdot v\dvol, \ \ \forall\ 0\leq v\in C_c^\infty(M-K).\]
Then there exists another compact subset $L\supset K$ of $M$ such that
\begin{align}\begin{split}\label{techeq}\int_{M-L}|u|^2\lt\exp(2b\rho_T)\dvol&\leq C_1 \left[ \int_{M-K}|w|^2\lt^{-1}\exp(2b\rho_T)\dvol \right. \\
&+ \left. \int_{L-K}|u|^2\lt\exp(2b\rho_T)\dvol \right] \end{split}\end{align}
for $C_1=\frac{8(1+b^2)}{(1-b^2)^2}$.
\end{lem}

\begin{cor}
If $w=cu$ for some $c>0$ and $T\geq \frac{2\sqrt{1+c}}{b\epsilon_f}$, then \[I(u):=\int_{M}|u|^2\exp(2b\rho_T)\dvol<\infty.\]
\end{cor}
\begin{proof}
With this choice of  $T$, $\lambda_T>{1+c}$ outside $K.$
Now replacing $\lambda_T$ with $\lambda_T-c$ and $w$ with $0$ in Lemma \ref{tech1}, we get
\begin{align*}
&\int_{M-L}|u|^2\exp(2b\rho_T)\dvol\leq \int_{M-L}|u|^2(\lt-c)\exp(2b\rho_T)\dvol\\
&\leq  C_1\int_{L-K}|u|^2\lt\exp(2b\rho_T)\dvol<\infty.
\end{align*}

\end{proof}

By refining the argument above, we have the following corollary which will be used in the proof of our Agmon estimate for eigenforms.

\begin{cor} \label{wL2e}
	If $0\leq u\in L^2(M)$, and $\Box_{Tf}u\leq (c+T|\nabla^2 f|)u$ for some $c>0$ and $T\geq \max \{(c+2c_f)C_2, \sqrt{\frac{1-b^2 +8c}{1-b^2}} /(b\epsilon_f) \}$, then \[I(u):=\int_{M}|u|^2\exp(2b\rho_T)\dvol \leq CT^2  \|u\|^2\]
	where the constant $C=C(C^L, C_L, r_0, b, c), L=\{p\in M:\rho_T(p)\leq 2\}, C^L>\max_L |\nabla f|^2, C_L>\max_L |\nabla^2 f|$.
\end{cor}
\begin{proof} Following the proof of Lemma \ref{tech1} given in Section 7.1, put $L=\{p\in M:\rho_T(p)\leq 2\}.$ Then we deduce
\begin{align*}\begin{split} \int_{M-L}|u|^2\lt\exp(2b\rho_T)\dvol&\leq C_1 \int_{L-K}|u|^2\lt\exp(2b\rho_T)\dvol \\
&+ C_2 \int_{M-K}(c+ T|\nabla^2f|) |u|^2\exp(2b\rho_T)\dvol \end{split}\end{align*}
for $C_1=\frac{8(1+b^2)}{(1-b^2)^2}$ as above, and $C_2=\frac{8}{1-b^2}$.	We split the second integral on the right hand side into two; the one over $L-K$ will be absorbed into the first term.
The second term is (we omit the volume form here)
\begin{align*}\begin{split}  C_2 \int_{M-L}c |u|^2\exp(2b\rho_T) &+ C_2 \int_{M-L}T|\nabla^2f| |u|^2\exp(2b\rho_T) \\
\leq  C_2 \int_{M-L}c |u|^2\exp(2b\rho_T) &+  \frac{C_2c_f}{T-cC_2/\epsilon_f} \int_{M-L}|u|^2(\lt-cC_2)\exp(2b\rho_T).   \end{split}\end{align*}
Combining the above one arrives at 
\begin{align*}\begin{split} \int_{M-L}|u|^2(\lt-cC_2) \exp(2b\rho_T)&\leq C_1 \int_{L-K}|u|^2(\lt + c + T |\nabla^2 f| ) \exp(2b\rho_T) \\
&+\frac{c_fC_2}{T-cC_2/\epsilon_f} \int_{M-L}|u|^2(\lt-cC_2)\exp(2b\rho_T).  \end{split}\end{align*}
Thus, for $T\geq (c+2c_f)C_2$, 
\begin{align*} \int_{M-L}|u|^2(\lt-cC_2) \exp(2b\rho_T)\leq 2C_1 (C^L b^2T^2 +c_LT +c)e^{4b} \|u\|^2, \end{align*}
where $C^L>\max_L |\nabla f|^2, C_L>\max_L |\nabla^2 f|$. If $T$ is also bigger than $\sqrt{\frac{1-b^2 +8c}{1-b^2}} /(b\epsilon_f)$, then $\lt > 1+ cC_2$ outside $L$. Hence
\begin{align*}
\int_{M-L}|u|^2\exp(2b\rho_T)\dvol &\leq \int_{M-L}|u|^2(\lt-cC-2)\exp(2b\rho_T)\dvol\\
&\leq   2C_1 (C^L b^2T^2 +c_LT +c)e^{4b} \|u\|^2,
\end{align*}
and consequently
\begin{align*}
\int_{M}|u|^2\exp(2b\rho_T)\dvol \leq  [ 2C_1 (C^L b^2T^2 +c_LT +c)+1]e^{4b} \|u\|^2,
\end{align*}
for $T\geq \max \{(c+2c_f)C_2, \sqrt{\frac{1-b^2 +8c}{1-b^2}} /(b\epsilon_f) \}$.
\end{proof}	
\begin{rem}\label{indepentoft}
It may seem that $C^L$ and $C_L$ depend on $T$ as  $L=\{p\in M:\rho_T(p)<r_0\}$. However,  notice that when $T$ becomes bigger, $L$ gets smaller. Hence we can choose $C^L>\max_{p\in L}|\nabla f|(p)$, $C_L>\max_{p\in L}|\nabla^2 f(p)|$, s.t. they are independent of $T.$
\end{rem}
\begin{lem}[De Giorgi-Nash-Moser Estimates]\label{moser}
For $r>0$, 
let $B_r(p)$ be the geodesic ball around $p$ with radius $r$ (in the metric $g$).  Let $0\leq u\in L^2(M)$, and   $\Delta u\leq cu$                
on $B_{2r}(p)$ in the weak sense for some constant  $c\geq 0.$ Then there exists constant $C_2>0$ depending only on the dimension $n$, the Sobolev constant, and $c$, such that
\[\sup_{y\in B_r(p)}u(y)\leq \frac{C_2}{r^{n/2}}\|u\|_{L^2(B_{2r}(p))}. \]
\end{lem}

With these preparation we are now ready to prove our first main estimate for the eigenforms of $\Box_{Tf}$.

\def\tb{\tilde{B}}
\begin{proof}[Proof of Theorem \ref{est}]
Consider an eigenform $\o$ of $\Box_{Tf}$. That is  $\Box_{Tf}\o=\mu(T) \o$, where the eigenvalue $\mu(T)$ satisfies $|\mu(T)|\leq c$ for some constant $c$. Then letting $u=g(\o,\o)^{1/2}$, by a straightforward computation using the Bochner's formula (for forms) and the Kato's inequality, we have
\[\Box_{Tf}u\leq (c+|R|+T|\nabla^2 f|)u,\]
where $|R|$ is the upper bound of curvature tensor. Hence by Corollary \ref{wL2e}, we have,  for $T\geq \max \{(c+|R|+2c_f)C_2, \sqrt{\frac{1-b^2 +8c+8|R|}{1-b^2}} /(b\epsilon_f) \}$,
\[I(u)=\int_{M}|u|^2\exp(2b\rho_T)\dvol \leq CT^2  \|u\|^2\]
where the constant $C=C(C^L, C_L, r_0, b, c, |R| )$.

Recall that the compact set $K$ is chosen so that (\ref{conditionK}) is satisfied. Hence by Proposition \ref{estihdg},
the conditions of Lemma \ref{moser} are satisfied for $u$ on $M-K$. Also, the Agmon distance $\rho_T(p)$ is the distance between $p$ and $K$ induced by $\gt$ and $L=\{p\in M:\rho_T(p)\leq 2\}.$ Suppose $p\in M-L$. Denote by $\tb_r(p)$ the $\gt$-geodesic ball around $p$ with radius $r$.  Set $l=\sup_{q\in \tb_2(p)}|T\nabla f|(q)$, and $r=1/(2l).$ Then one  can easily verify that $B_{2r}(q)\subset \tb_2(p),\ \forall q\in \tb_1(p).$

Choose $q_0\in \overline{\tb_2(p)}$ so that $|T\nabla f|(q_0)\in (l/2,l].$
By Lemma \ref{tech1} and de Giorgi-Nash-Moser estimate Lemma \ref{moser}, we have
\begin{align*}
|u(p)|^2\exp(2b\rho_T(p))&\leq \frac{C_2}{r^{n}}\|u\|^2_{L^2(B_{2r}(p))}\exp(2b\rho_T(p))\\
&\leq\frac{C_3}{r^{n}}\int_{\tb_2(p)}|u|^2(q)\exp(2b\rho_T(q))\dvol \\
&\leq {C_4}|T\nabla f(q_0)|^{n}I(u).
\end{align*}
We will prove that \begin{equation}\label{tbp}|\nabla f(q_0)|^2 \leq C_5\exp(\frac{2c_f}{bT} \rho_T(q_0))\end{equation} in Lemma \ref{remm}. Hence,
\[ |\nabla f(q_0)|^2 \leq C_6\exp(\frac{2c_f}{bT}\rho_T(p)) \leq C_6\exp(\epsilon \rho_T(p)) \]
for any small $\epsilon$, provided $T\geq \frac{2c_f}{b\epsilon}$.
It follows then that,
\[|u(p)|^2\leq C_6I(u)T^n\exp(-2a\rho_T(p)),\]
for any $a<b$ provided $T\geq \frac{nc_f}{b(b-a)}$. Hence if $T\geq \max \{(c+|R|+2c_f)C_2, \sqrt{\frac{1-b^2 +8c+8|R|}{1-b^2}} /(b\epsilon_f),  \frac{nc_f}{b(b-a)} \}$,
\[|u(p)|^2\leq C_6C(C^L, C_L, r_0, b, c, |R| )T^{n+2} \exp(-2a\rho_T(p)) \|u\|^2.\]
\end{proof}
\begin{rem}
The proof above gives the  inequality for $p\in M-L=\{\rho_T(p)>r_0\}$ for some constant $r_0$ independent of $T$, which is what we needed for later applications. For $p\in L$, using the same reasoning as in Remark \ref{indepentoft}, there exist constant $C>0$, which is independent of $T$, such that 
\begin{equation}
\Delta u\leq CTu.
\end{equation} 
for all $p\in L$. Therefore via Moser iteration as in Lemma \ref{moser} and similar arguments as above, one can show that
\[|u|^2(p)\leq C'T^{n}\|u\|_{L^2}^2\leq C'\exp(2a)T^{n}\exp(-a\rho_T)\|u\|_{L^2}^2.\]
\end{rem}
\section{Morse inequalities}\label{MorseInequality}
In this and the next section, we assume that $f$ is a Morse function on $M$, and $T\geq T_0$. In fact, we assume that in a neighborhood $U_x$ of critical points $x$ of $f$, we have coordinate system $z=(z_1,...,z_n),$ such that
\begin{align}\begin{split}\label{flatMetric}
f=-z_1^2-...-z_{n_f(x)}^2+z_{n_f(x)+1}^2+...+z_n^2,
g=dz_1^2+...+dz_n^2.
\end{split}\end{align}
This is a generic condition.
Without loss of generality we assume that $U_x$ is an Euclidean open ball around $x$ with radius $1.$ Also, these open sets are disjoint.

Let $F_{Tf}^{[0,1],*}$ be the space spanned by the eigenforms of $\Box_{Tf}$ with eigenvalue lying in $[0,1].$ By Theorem \ref{discrete}, $F_{Tf}^{[0,1],*}$ is finite dimensional. Recall that $m_i$ denotes the number of critical points of $f$ with Morse index $i.$ We have the following Proposition:
\begin{prop}\label{mo}
There exists $T_3>T_0$ big enough, so that whenever $T\geq T_3,$ the number of eigenvalues (counted with multiplicity) in $[0,1]$ of $\Box_{Tf}|_{\Omega^i_{(2)}(M)}$ equals $m_i$. I.e. $\dim F_{Tf}^{[0,1],*}=m_i$.
\end{prop}
\begin{rem}
See the definition of $T_0$ in (\ref{T0}). Also recall that, if $(M,g,f)$ is strongly tame, $T_0=0.$
\end{rem}
The proof of Proposition \ref{mo} follows from that of Proposition 5.5 in \cite{zhang2001lectures}, except for the proof of the following proposition:
\begin{prop}
There exist constants $C > 0$, $T_4 > 0$ such that for any
smooth form $\phi \in \Omega^*_{(2)}(M)$ with ${\rm supp}(\phi) \subset M-\cup_{x\in \Crit{f}}U_x$ and $T \geq T_4$, one has
\[\|\Box_{Tf}\phi\|_{L^2}\geq CT\|\phi\|_{L^2}.\]

Here ${\rm supp}(\phi)$ denotes the support of $\phi.$
\end{prop}
\begin{proof}
Since $f$ is well tame, there exist $\delta_1,\delta_2>0,$ s.t. $|\nabla f|\geq \delta_1,$ also $|\nabla^2f|\leq \delta_2 |\nabla f|^2$ on $M-\cup_{x\in \Crit{f}}U_x$.
Then our proposition follows from the same argument in Proposition 4.7 of \cite{zhang2001lectures}.
\end{proof}

On the other hand, $(F_{Tf}^{[0,1],*}, d_{Tf})$  form a complex, the so called Witten instanton complex, whose cohomology is $H^*_{(2)}(M,d_{f})$ by Theorem \ref{kodairahodge}.
As a result, our Theorem \ref{morse} (the strong Morse inequalities) follows from Proposition \ref{mo} and our Hodge theorem when $T>T_3$. For the case of $T\in (T_0,T_3]$, see section \ref{tae}.
\section{Thom-Smale theory}\label{tsw}

In this section, we assume that $K$ is a compact subset of $M$, $\epsilon>0$ is small enough (to be determined later), $T_5=T_5(\epsilon)$ is big enough, such that outside $K$, we have 
\begin{equation}T|\nabla^2f|\leq \epsilon\, T^2|\nabla f|^2, \label{2tc} \end{equation} 
  provided $T\geq T_5$.

Moreover, we make a more judicious choice of $K.$  Fix any $p_0\in M$. Set
 \begin{eqnarray*} D & = & \sup_{x,y\in \Crit(f)}{\tilde{d}_T(x,p_0)+\tilde{d}_T(x,y)}  \\ 
 & = & \sup_{x,y\in \Crit(f)} bT(|f(x)-f(p_0)|+|f(x)-f(y)|),\end{eqnarray*}
 where $\tilde{d}_T$ is the distance function induced by $\gt$ (note that the Agmon distance $\rho_T(x)$ is the same distance function but between $x$ and a compact subset $K$). The second equality follows from the claim in the proof of Lemma \ref{rho}.
 We choose $K$ so that that $$\tilde{B}_{D+1}(p_0)=\{p\in M: \tilde{d}_T(p,p_0)\leq D+1\}\subset K^\circ$$ where $K^\circ$ denotes the interior of $K.$

\begin{rem}
We can take $T_5=0$ if $(M,g,f)$ is strongly tame.
\end{rem}

Now we set 
\begin{equation}\label{T0}
  T_0=\max\{T_1,T_2,T_5\}.  
\end{equation}

(Note that the definition of $T_0$ does not involve $T_3$; Cf. (\ref{t1}) and (\ref{t2}) for the description of $T_1$ and $T_2$.)
 Before defining the Thom-Smale complex, there is still a subtle issue for noncompact cases. That is, the gradient vector field $-\nabla f$ may not be complete, i.e., its flow curves may not exist for all time. But notice that if we rescale the vector field by some positive function, we actually get the reparameterization of flow curves.

For this purpose, we fix a positve smooth function $F$ such that
\begin{equation*}
    F|_{M_K}=\frac{1}{T^2|\nabla f|^2}.
\end{equation*}

Then we have 
\begin{lem}\label{complete}
$-F\nabla f$ is a complete vector field.
\end{lem}
\begin{proof}
Let $\tilde{\Phi}^t$ be the flow generated by $-F\nabla f$.
We show that for any $p\in M$, there exists a universal $\epsilon_0>0$, s.t. $\tPhi(p)$ is well defined on $(-\epsilon_0,\epsilon_0)$. Hence, $-F\nabla f$ is complete. 

Let $L:=\{p\in M:\td(p,K)\leq 1\}$. It suffices to show that for any $p\in M-L,$ $\tPhi(p)$ is well defined on $(-1,1)$, since $L$ is compact.

But on $M-K$, $F^{-1}g(-F\nabla f,-F\nabla f)=\frac{1}{T^2}$, and $(M,F^{-1}g)$ is complete, hence $\tPhi(p)$, $t\in(-1,1)$ is a geodesic (See Lemma \ref{rho}) inside $M-K$  for $p\notin L$.

\end{proof}


Let $x$ be a critical point of the Morse function $f$, $W^s(x)$ and $W^u(x)$ be the stable and unstable manifold of $x$ with respect to flow $\tilde{\Phi}^t$ defined in Lemma \ref{complete} (See Chapter 6 in \cite{zhang2001lectures}). We will further assume that $f$ satisfies the Smale transversality condition, namely $W^s(x)$ and $W^u(y)$ intersect transversally. Then the Thom-Smale complex $(C_*(W^u),\partial)$ is defined by
\[C_*(W^u_f)=\oplus_{x\in \Crit(f)}\R W^u(x),\]
and
\[C_i(W^u)=\oplus_{x\in \Crit(f),n_f(x)=i}\R W^u(x).\]
To define the boundary operator, let
$x$ and $y$ be  critical points of $f$, with $n_f(y) = n_f(x) - 1$.

For $x\in \Crit(f),$ set
\[\patt W^u(x)=\sum_{y\in \Crit(f),n_f(y)=n_f(x)-1}m(x,y)W^u(y).\]
Here the integer $m(x,y)$ is the signed counts of the flow lines in $W^s(y)\cap W^u(x),$.

 \begin{rem}\label{fgoestoinfinity}
With our nice choice of $K$ and $F$, it is easy to see that for any $x,y\in \Crit(f)$, $W^s(x)\cap W^u(y)\subset K^\circ$. Moreover, for any $p\in W^u(x)-K$, the curve $\{\tPhi(p):t\geq 0\}\cap K=\emptyset$ and $\lim_{t\to\infty}|f|(\tPhi(p))\approx\lim_{t\to\infty}\rho_T(\tPhi(p))=\infty.$ Moreover, since $\tilde{d}_T(p,p_0)=T\tilde{d}_1(p,p_0)$, we can actually choose $K$, such that it is independent of $T.$ Thus, just like the compact case, by the transversality, $m(x,y)$  is well defined. 
 \end{rem} 

We will prove in Section 7.3 that under our tameness condition, $\patt^2=0.$ Thus, $C_*(W^u,\patt)$ is a complex.

\def\P{\mathcal{P}}
\def\J{\mathcal{J}}
Recall that the Witten instanton complex $F_{Tf}^{[0,1],*}$ is the finite dimensional space generated by the eigenforms of $\Box_{Tf}$ with eigenvalue lying in $[0,1]$.  By the discussion in the previous subsection, the cohomology of the Witten instanton comple is $H^*_{(2)}(M,d_{f})$.

To prove Theorem \ref{wts}, we now consider the following chain map $\J: (F_{Tf}^{[0,1],*},d_{Tf})\mapsto C^*((W^u)',\patt').$ Here $C^*((W^u)',\patt')$ denote the dual chain complex. Let $W^u(x)'$ be the dual basis of $W^u(x).$  Then
 
$$\J\o= \sum_{x\in\Crit(f)}W^u(x)'\int_{W^u(x)}\exp(Tf)\o.$$ 
However there is a technical issue here we need to address. 
When $\overline{W^u(x)}$ is compact, the integral $\int_{W^u(x)}\exp(Tf)\o$ is clearly well defined, but $\overline{W^u(x)}$ here may be noncompact.
We will be content here only with the wel-definedness of the map and leave the proof that $\J$ is indeed a chain map to Section 7.3.

\def\dr{B_r^{n_f(x)}(x)}
Let $r>0$ small enough, $B_r^{n_f(x)}(x)\subset K$ be the $n_f(x)$-dimensional ball in $W^u(x)$ with center $x$ and  radius $r$ with respect to metric $g.$
As before, let $\tPhi$ be the flow generated by $-F\nabla f$. Then  $W^u(x)=\lim_{t\to\infty}\tPhi(\dr).$

Therefore, for any $\o\in F^{[0,1],*}_{Tf}$
\begin{align*}
&|\int_{W^u(x)}\exp(Tf)\o|=|\lim_{t\to\infty}\int_{(\tPhi)(\dr)}\exp(Tf)\o|\\
&\leq C \exp(Tf(x))\lim_{t\to\infty}\int_{\dr}|(\tPhi)^*\o||\det((\tPhi)_*)|\dvol
\end{align*}
\def\nf{n_{f}(x)}
The well definedness of $\J$ is now reduced to the following two technical lemmas, as well as Theorem \ref{est} and the well tameness of $(M,g,f)$.

\begin{lem}\label{vol}
Fix any $y\in\dr-\tilde{\Phi}^{-t}K,$ we have
\begin{align*}  |\Phi^t_*(y)| & \leq C_7(T)\exp(\epsilon\rho_T(p)) \\
\end{align*}
Hence,
\begin{align*}  |\det(\Phi^t)_*(y)|\leq C_7(T)\exp(n_f(x)\epsilon\rho_T(p))\end{align*}
Here $C_7$ is a constant independent of $y.$
\end{lem}

\begin{proof}
Let $e$ be a unit tangent vector of $W^u(x)$ at $y,$ Extend $e$ to a local unit vector field near $y$ via parallel transport along radial geodesics. Denote 
\begin{equation}\label{yf}
 Y_f=-F\nabla f 
\end{equation}
Noting that from (\ref{2tc})

\[|\nabla_{\tPhi_*e}(\tPhi)_* (Y_f)|\leq \frac{\epsilon}{T}|(\tPhi)_*e|,\]
we have
\begin{align*}&|\frac{\partial}{\partial t}g((\tPhi)_*e(y),(\tPhi)_*e(y))|\\
&=2|g(\nabla_{(\tPhi)_*e(y)}(\tPhi)_*Y_f,(\tPhi)_*e(y))|\\
&\leq \frac{2\epsilon}{T}|g((\tPhi)_*e(y),(\tPhi)_*e(y))|.
\end{align*}

By a classical result in ODE, we have
\[g((\tPhi)_*e(y),(\tPhi)_*e(y)) \leq C_8\exp(\frac{2\epsilon t}{T}).\]

Now our lemma follows from teh following, Lemma \ref{rho}.
\end{proof}
\def\diam{\mathrm{diam}}

\begin{lem}\label{rho}
Suppose $t>0$ is big enough, $y\in \dr-\tilde{\Phi}^{-t} K$. Then there exists a constant $C_9=C_9(K,|\nabla f|^2|_K)>0,$ s.t.
\[|\rho_T(\tPhi(y))-\frac{t}{T}|<C_9T.\]
\end{lem}

\begin{proof}

For any $y\in \dr-\tPhi K$, we claim that $\tilde{\Phi(y)}^s,s\in[0,t]$ is one of the shortest smooth curve on $(M,\gt)$ connecting $y$ to $\tPhi(y).$

Granted, since $\gt(Y_f(p),Y_f(p))=\frac{1}{T^2}$ for all $p\notin K$, the claim gives
\begin{align*}
&|\rho_T(\Phi^t(y))-\frac{t}{T}|=|\td(\Phi^t(y),K)-\td(\Phi^t(y),y)| \\
&\leq T\sup_{p\in K}|\nabla f| \diam(K),
\end{align*}
where $\td$ is the distance induced by $\gt$, $\diam(K)$ is the diameter of $K$ with respect to metric $g.$ 

We now prove the claim (See \cite{Puits} Lemma A 2.2 for another proof):

\def\tphi{\Tilde{\Phi}}
\def\te{\tilde{e}^T}
\def\tna{\tilde{\nabla}^T}

First, Let's show that $\gamma(s):=\tphi^s(y), s\in[0,r]$ is a geodesic for any $r>0$:

Let $\te_1(s),...,\te_n(s)$ be a local orthomormal frame on $\gamma$ with $\te_1=-\tna f=\gamma'$. In order to prove $\gamma''=0,$ it suffices to prove $\gt(\gamma'',\te_i)=0,i\geq 2.$

Let $\tna$ be the Levi-Civita connection induced by $\gt,$ then
\begin{align*}
\gt(\gamma'',\te_i)&=\gt(\tna_{\tna f}\tna f,\te_i)\\
&=-\gt(\tna f,[\tna f,\te_i])\\
&=-[\tna f,\te_i]f\\
&=-\tna f\gt(\te_i,\tna f)+\te_i\gt(\tna f,\tna f)\\
&=0.
\end{align*}
We now prove that $\gamma$ is the shortest geodesic connecting $p$ and $\gamma(r)$ in $(M,\gt)$, for all $r>0$:

 Assume that $\sigma$ is another normal geodesic connecting $p$ and $\gamma(r)$ induced by $\gt.$ Then $\gt(\sigma'(0),\tna f(p))<1.$ Set $a(s)=f\circ\gamma(s),b(s)=f\circ\sigma(s),$ then we have $a(0)=b(0),$ and $a'(s)=-1,b'(s)=-\gt(\sigma'(s),\tna f\circ\gamma(s))\geq-1.$  Hence by a comparison theorem in ODE, we must have $a(s)\leq b(s).$
Assume that $\sigma(r')=\gamma(r),$ then we can see $r'=Length(\sigma),$ also we have $a(r')\leq b(r')=a(r).$ Since $a$ is decreasing, we must have $r'\geq r.$

By now, we can see that $\tilde{\Phi}^s(y),s\in [0,t]$ is one of shortest geodesic connecting $y$ and $\tPhi(y).$
\end{proof}

We now note the following lemma which plays an important role in estimating the eigenforms previously.
\def\tg{\tilde{g}_T}
\begin{lem}\label{remm}
Suppose $T\geq T_0.$ Then for any $q\in M$, there exists $C>0$, such that
\[|\nabla f|^2(q)\leq C\exp( \frac{c_f}{bT} \rho_T(q))\]
\end{lem}

\begin{proof}

Let $\gamma:[0,\rho_T(q)]\mapsto M$ be a normal minimal $\tg$-geodesic connecting $p_0$ and $q$. Then we have
$g(\gamma',\gamma')=\frac{1}{b^2T^2|\nabla f|^2}$ outside $K.$ As $K$ is compact and $|\nabla f|$ is bounded on $K$, we will assume without loss of generality that $\gamma$ lies outside of $K$.

Let $h(t)=|\nabla f|^2\circ \gamma,$ then
\[h'(t)=g(\nabla_{\gamma'}\nabla f,\nabla f)\leq \frac{1}{bT}|\nabla^2f|\leq \frac{c_f}{bT} |\nabla f|^2= \frac{c_f}{bT}  h(t),\]

Hence $|\nabla f|^2(q)\leq C\exp( \frac{c_f}{bT}  \rho_T(q)). $
\end{proof}

Here we are gives a direct proof of the isomorphism of $H^*_{(2)}(M,d_{Tf})$ and $H^*_{dR}(M,U_c)$ under the assumption that $f$ is proper.

\begin{thm}\label{thm}
 Assume that $f$ is proper. Set $I=\inf_{p\in K}f(p),S=\sup_{p\in K}f(p)$ and fix $c>|I|+|S|+2$. Then for $U_c=\{p\in M:f(p)<-c\},$
$(\Omega_{(2)}^*(M),d_{Tf})$
  and $(\Omega^*(M,U_c),d)$ are quasi-isomorphic.
\end{thm}
\begin{proof}
We may as well set $K=f^{-1}[I,S].$
Motivated by \cite{farbershustin00polyform}, consider $(Cone^*,d_C)$, where $Cone^j=\Omega^j(M)\oplus \Omega^{j-1}(U_c),$
\[d_C(\phi,\phi')=(d\phi,(-1)^j(-d|_{U_c}\phi'+\phi|_{U_c})).\]
Then $(Cone^*,d_c)$ and $(\Omega^*(M,U_c),d)$ are quasi-isomorphic.
\def\bPhi{{\bar\Phi}}
\def\bnf{-\frac{\nabla f}{|\nabla f|}}
\def\L{\mathcal{L}}

Set $U_c'=\{p\in M:f(p)>c\}$, $U=U_c\cup U_c'$. Let $\bar\Phi^t$ be the flow in $U$ generated by $X_f=-\frac{\nabla f}{b^2T^2|\nabla f|^2}$ on $U_c$, and $X_f=\frac{\nabla f}{b^2T^2|\nabla f|^2}$ on $U_c'$.

Define a map $\L:F_{Tf}^{[0,1],j}\mapsto Cone^j,$
\begin{equation*}
\o\mapsto (\exp(Tf)\o,(-1)^j\int_0^\infty(\bPhi^s)^*(\exp(Tf)\iota_{X_f}\o)ds)
\end{equation*}
By Theorem \ref{est}, and similar argument in Lemma \ref{rho}, we can see that $|(\bPhi^s)^*\iota_{X_f}\o|\leq C\exp(-\int_0^saTdt)\leq C\exp(-aTs).$
Hence, $\L$ is well defined.

$\L$ is a chain map, since

 \begin{align*}
\L(d_{Tf}\o)&=(\exp(Tf) d_{Tf}\o,(-1)^{j+1}\int_0^\infty(\bPhi^s)^*(\exp(Tf)\iota_{X_f} d_{Tf}\o)ds)\\
&=(\exp(Tf) d_{Tf}\o,(-1)^{j+1}\int_0^\infty(\bPhi^s)^*(\exp(Tf)\iota_{X_f} d_{Tf}\o)ds)\\
&=(d(\exp(Tf)\o),(-1)^{j+1}\int_0^\infty(\bPhi^s)^*(\iota_{X_f}d(\exp(Tf)\o)ds))\\
&=(d(\exp(Tf)\o),(-1)^{j+1}\int_0^\infty(\bPhi^s)^*(L_{X_f}(\exp(Tf)\o)+d\iota_{X_f}(\exp(Tf)\o ds)\\
&=(d(\exp(Tf)\o),(-1)^{j+1}\int_0^\infty\frac{d}{ds}(\bPhi^s)^*(\exp(Tf)\o)+d(\exp(Tf)\iota_{X_f}\o) ds)\\
&=(d(\exp(Tf)\o),(-1)^{j+1}(-\exp(Tf)\o+d\int_0^\infty (\bPhi^s)^*(\exp(Tf)\iota_{X_f}\o) ds)\\
&=d_C\L(\o).
\end{align*}

 Hence $\L$ induces a homomorphism (still denote it by $\L$) between $H^*(\Omega_{(2)}^{\bullet}(M),d_{Tf})$ and $H^*(Cone^\bullet,d_C).$ The proof of the fact that $\L$ is a bijection is tedious, which will be given in Subsection \ref{bijec}.
\end{proof}
\section{An application of Theorem \ref{wts}}  \label{atmwb}
\def\M{\tilde{M}}
Let $(M,g)$ be an oriented, compact Riemannian manifold with boundary $\pat M,$ and near the boundary. Let $\M:=M\cup T,$ where $T\cong \pat M\times [1,\infty)$. Then we can extend the metric $g$ to $\M$, s.t. near the infinity, the metric $g$ on $\M$ is of product type, i.e. $g_{\partial M}+dr^2$. It's easy to see that $(\M,g)$ has bounded geometry.
\def\f{\tilde{f}}
We have the following technical lemma
\begin{lem}
Given a transversal Morse function $f$, a partition of boundaries $N^+=N^+_1\sqcup N^+_2,$ $N^-=N^-_1\sqcup N^-_2.$ We are able to extend $f$ to a function $\f$ on $\M,$ s.t.
\begin{enumerate}
\item $|\f|(x)\to\infty$ as $x\to\infty$;
\item $\f<0$ on $(N_1^+\sqcup N_2^-)\times (a,\infty];$
\item $\f$ has critical points $\Crit^{*}(\f)=Cirt^{\circ,*}(f)\cup Cirt^{+,*}_{N^+_1}(f)\cup Cirt^{-,*}_{N^-_1}(f)$;
\item $(\M,g,\f)$ is well tame.
\end{enumerate}
\end{lem}
\begin{proof}
Use notation $x=(x',r),x'\in\pat M,r\in(0,1]$ to denote $x\in\pat M\times (0,1].$
Since $f$ is a transversal Morse function, there exists $s_0<1$, s.t. $|\frac{\pat f}{\pat r}(x',r)|\neq0$ on $\pat M\times(s_0,1].$

Hence, by considering the Taylor expansion of $f(x',r)$ with respect to $r,$ there is a smooth function $\theta$ on $\pat M\times(s_0,1],$ s.t.
\begin{enumerate}
\item $f(x',r)=f(x',1)+\frac{\pat f}{\pat r}(x',1)\theta(x',r)$;
\item $\theta(x',r)=(r-1)+o((r-1))$ near $\pat M\times\{1\};$
\item $\frac{\pat \theta}{\pat r}(x',r')=1+o(1).$
\end{enumerate}
 Assume that $s_0$ is close enough to $1$, s.t. $\theta(x',r)<\min\{-(r-1)^2,1/2(r-1)\},\frac{\pat \theta}{\pat r}(x',r')>0$ on $\pat M\times(s_0,1].$

Let $\eta_1$ be a smooth function on $(-\infty, \infty),$ s.t.
 \def\Z{\mathbb{Z}}
 \begin{enumerate}
\item $0<\eta<1$ on $(s_0,1)$, and $\eta\equiv0$ on $(-\infty,s_0)$, $\eta\equiv1$ on $(1,\infty)$;
\item $\eta'(r)>0,\forall r\in(s_0,1)$;
\end{enumerate}
Then $\forall x'\in N^+_1\cup N^-_1$, let \[\f(x',r)=f(x',r)=f(x',1)+\frac{\pat f}{\pat r}(x',1)((1-\eta(r))\theta(x',r)+\eta(r)(r-1)^2);\] $\forall x'\in N^+_2\cup N^-_2$, let \[\f(x',r)=f(x',r)=f(x',1)+\frac{\pat f}{\pat r}(x',1)((1-\eta(r))\theta(x',r)+\eta(r)1/2(r-1)).\]
It's easy to verify that $\f$ satisfy our conditions.
\end{proof}

Since the Thom-Smale complex $(C^*(W^u),\patt')$ of $\f$ induces a differential operator $\patt'$ on $\Crit^*(f)$, Theorem \ref{bdy} follows easily from Theorem \ref{wts}.

\section{The Agmon Estimate}\label{tae}
\def\L{\mathcal{L}}

In this section we carry out the main technical estimates of the paper.
\subsection{Proof of Lemma \ref{tech1}}
\begin{proof}
Our proof is adapted from Theorem 1.5 in \cite{agmon2014lectures}.


Let $L=\{p\in M:\rho_T(p)\leq 2\}.$
\def\ek{\eta_k}
Let $\eta_k\in C^\infty_c(\R)$ ($k$ large enough) be a smooth bump function such that
\begin{equation*}
\eta_k(t)=
\begin{cases}
0, \mbox{ If $|t|<1$ or $|t|>k+1$;}\\
1, \mbox{ If $|t|\in(2,k)$},
\end{cases}
\end{equation*}
and $|\eta_k'(t)|\leq 2,$ $\eta_k(t)\in[0,1],\forall t\in\R.$

\def\rjt{\rho_{T,j}}
\def\phikj{\varphi_{k,j}}
\def\ltj{\lambda_{T,j}}
Set $\rho_{T,j}=\min\{\rho_T,j\}$, and
\[\ltj=\begin{cases}
\lt, \mbox{ if } \rho_T<j,\\
0, \mbox{ otherwise}
\end{cases}.\]Clearly $|\nabla \rho_{T,j}|^2=\ltj$ a.e. and $\lt\geq \ltj.$

Now set $\varphi_{k,j}=(\eta_k\circ\rho_T)\exp(b\rho_{T,j})$.
 Then by assumption, we have
\[\int_{M}\nabla u\nabla(\phikj^2u)+\lt (u\phikj)^2\dvol\leq \int_Mw\phikj^2u\dvol. \]
Noting that $\nabla u\nabla(\phikj^2u)=|\nabla (\phikj u)|^2-|\nabla \phikj |^2u^2\geq -|\nabla \phikj |^2u^2, $ we have
\begin{equation}\label{eq7}
\int_{M-K}(\lt|u\phikj|^2-|u|^2|\nabla\phikj|^2)\dvol\leq \int_{M-K}wu\phikj^2\dvol.
\end{equation}
Since (we now omit the volume form $\dvol$ in what follows) \[\int_{M-K}wu\phikj^2 \leq \frac{1}{1-b^2}\int_{M-K}(\lt)^{-1}w^2\phikj^2+\frac{1-b^2}{4}\int_{M-K}\lt u^2\phikj^2,\]
and
\begin{align*}
|\nabla \phikj|^2&\leq \frac{1+b^2}{2} (\eta_k\circ\rho_T)^2|\nabla \rjt|^2\exp(2b\rjt) + \frac{1+b^2}{1-b^2}(\eta_k'\circ\rho_T)^2|\nabla \rho_T|^2\exp(2b\rjt) \\
&=\frac{1+b^2}{2}(\eta_k\circ\rho_T)^2\ltj \exp(2b\rjt)+\frac{1+b^2}{1-b^2}(\eta_k'\circ\rho_T)^2\lt\exp(2b\rjt),\end{align*}
by (\ref{eq7}), we have
\def\ek{\eta_k\circ\rho_T}
\def\ekp{\eta_k'\circ\rho_T}
\begin{align}
\begin{split}
\label{sim}
&\frac{3+b^2}{4}\int_{M-K}\lt (\eta_k\circ\rho_T)^2u^2\exp(2b\rho_{T,j})
-\frac{1+b^2}{2}\int_{M-K}\ltj (\eta_k\circ\rho_T)^2u^2\exp(2b\rho_{T,j}) \\
&\leq \frac{1}{1-b^2}\int_{M-K}w^2(\ek)^2\lt^{-1}\exp(2b\rho_{T,j}) \\
&+\frac{1+b^2}{1-b^2}\int_{M-K}u^2(\ekp)^2\lt\exp(2b\rho_{T,j}) \\
&\leq \frac{1}{1-b^2}\int_{M-K}w^2\lt^{-1}\exp(2a\rho_{T}) \\
&+2\frac{1+b^2}{1-b^2}\int_{L-K}u^2\lt\exp(2b\rho_{T}) +2\frac{1+b^2}{1-b^2}\int_{\tb_{k+1}-\tb_{k}}u^2\lt\exp(2bj)
\end{split}\end{align}
Letting $k\to\infty,$ by the monotone convergence theorem and the fact that $\int_{M}\lt u^2<\infty$, we have
\begin{align*}
&\frac{3+b^2}{4} \int_{M-L}\lt u^2\exp(2b\rho_{T,j})-\frac{1+b^2}{2}\int_{M-L}\ltj u^2\exp(2b\rho_{T,j})\\
&\leq\frac{1}{1-b^2}\int_{M-K}w^2\lt^{-1}\exp(2b\rho_{T})+2\frac{1+b^2}{1-b^2}\int_{L-K}u^2\lt\exp(2b\rho_{T}).
\end{align*}
Now let $j\to\infty.$ By the monotone convergence theorem again, we arrive at (\ref{techeq}), as desired.
\end{proof}
\subsection{Proof of Lemma \ref{moser}}
\begin{proof}
Our proof is a standard argument of Moser iteration.

Assume that $\rho_T(p)$ is big enough and $r$ is small enough, s.t. $B_{2r}(p)\cap K=\emptyset.$ Then on $B_{2r}(p)$, we have
\begin{equation}\label{delta}cu\geq \Box_{Tf}u\geq \Delta u\end{equation}
weakly.

Set $r_1=2r, r_{k+1}=r_k-(1/2)^kr, n_k=(n/(n-2))^{k-1}$

Let $\eta_k\in C^\infty_c(B_{2r})$ be bump functions s.t.
\[\eta_k=\begin{cases}
1 \mbox{ on $B_{r_{k+1}}$,}\\
0 \mbox{ on $B_{2r}-B_{r_{k}}$,}
\end{cases}\]
and $|\nabla \eta_k(q)|<\frac{2}{r_{k+1}-r_k},$ $\eta_k(q)\in [0,1],\forall q\in B_{2r}.$
\def\um{u_{m}}
\def\u{u_{\xi}}

Set $u_{m}=\min\{u,m\}$, and $\phi_1=\eta_1^2u_m\in H_0^1(B_{2r}).$ Notice that $\phi_1=0$ and $\nabla \phi_1=0$ in $\{u\geq m\}.$
Hence, by (\ref{delta}), we have
\begin{align*}
&\int_{B_{r_1}}c(u_m)^2dvol\geq\int_{B_{2r}}cu\phi_1dvol\geq\int_{B_{2r}}\nabla u\nabla \phi_1 dvol\\
&=\int_{B_{2r}}\eta_1^2|\nabla u_m|^2+2\eta_1\nabla \eta\nabla u_mu_mdvol\\
&\geq \int_{B_{2r}}\eta_1^2|\nabla u_m|^2-1/2\eta_1^2|\nabla u_m|^2-2|\nabla \eta u_m|^2dvol\\
&\geq \int_{B_{2r}}\eta_1^2|\nabla u_m|^2-1/2\eta_1^2|\nabla u_m|^2-2|\nabla \eta u_m|^2dvol\\
&\geq 1/2\int_{B_{r_2}}|\nabla u_m|^2dvol-4/(r_2-r_1)^2\int_{B_{r_1}}| u_m|^2dvol\\
\end{align*}
Hence, we have
\[\int_{B_{r_2}}|\nabla u_m|^2dvol\leq (2c+8/(r_1-r_2)^2)\int_{B_{r_2}}c(u_m)^2dvol\leq C(n)/(r_1-r_2)^2\int_{B_{r_2}}c(u_m)^2dvol.\]
By Sobolev inequality,
\[(\int_{B_{r_1}}| u_m|^{2n_2}dvol)^{1/n_2}\leq C(n)/(r_1-r_2)^2\int_{B_{r_2}}c(u_m)^2dvol.\]
That is
\[\|u_m\|_{L^{2n_2}(B_{r_2})}\leq (C(n)/(r_1-r_2))\|u_m\|_{L^{2n_1}(B_{r_1})}\]
Let $m\to\infty,$ we have
\[\|u\|_{L^{2n_2}(B_{r_2})}\leq (C(n)/(r_1-r_2))\|u\|_{L^{2n_1}(B_{r_1})}\]
Consider $\phi_k=\eta_k^2(\um^{2n_k-1})\in H_0^1(B_{2r}).$ By the same arguments as above, we have
\[\|u\|_{L^{2n_{k+1}}(B_{r_{k+1}})}\leq (C(n)/(r_k-r_{k+1}))^{1/(n_k)}\|u\|_{L^{2n_k}(B_{r_k})}.\]
As a consequence,
\begin{align*}&\|u\|_{L^\infty(B_{r})}=\lim_{k\to\infty}\|u\|_{L^{2n_k}(B_{r_k})}\\
&\leq C\Pi_{k=1}^{\infty}(C(n)/(r_k-r_{k+1}))^{1/(n_k)}\|u\|_{L^2(B_{2r})}\\
&=C(C(n)/r)^{(\sum_{k=1}^{\infty}{1/(n_k)})}2^{\sum_{k=1}^\infty k/n_k}\|u\|_{L^2(B_{2r})}\\
&\leq C/r^{n/2}\|u\|_{L^2(B_{2r})}.
\end{align*}
\end{proof}

We state two Lemmas that would be needed shortly,
\begin{lem}\label{moser1}
 Suppose $u,w\in L^2(M)$, s.t. $\Box_{Tf}u\leq w$ in weak sense (Here we assume $u\geq 0.$). For $r>0$ small enough, $p\notin L$, let $B_r(p)$ be the geodesic ball around $p$ with radius $r$ induced by $g$. Then there exist $C_2>0$, s.t.
\[\sup_{y\in B_r(p)}u(y)\leq \frac{C_2}{r^{n/2}}(\|u\|_{L^2(B_{2r}(p))}+\|w\|_{L^2(B_{2r}(p))}),\]
where $C_2$ is a constant that depends only on dimension $n.$
\end{lem}
\begin{proof}
The proof is actually similar to the proof Lemma \ref{moser}, but a little more complicated.
See Theorem 4.1 in \cite{han2011elliptic} for a reference.
\end{proof}
By the same argument as the proof of Theorem \ref{est}, we have

\begin{lem}\label{epde}

Let $(M,g,f)$ be well tame, $w\in L^2(M)$, s.t.
\[\int_{M}\lt^{-1}|w|^2\exp(a''\rho_T)dvol<\infty\]
for some $a''\in(0,b).$
If $\phi\in L^2(M)$ is a weak solution of $\Box_{Tf}\phi\leq w,$
 then
\[|\phi(p)|\leq C\exp(-a''\rho_T(p)).\]

\end{lem}
\subsection{On the Thom-Smale complex}\label{ThomSmaleComplex}

First, let's recall the situation of the compact case. The following is a restatement of Proposition 6 in \cite{laudenbach1992thom}.
\begin{prop}\label{laudenbach}
Let $(N,g)$ be a compact Riemannian manifold, $f$ be a Morse function. Assume that $(N,g,f)$ satisfies Thom-Smale transversality condition. Then, for any critical point $x\in \Crit(f)$ with Morse index $n_f(x)$, any $\phi\in \O^{n_f(x)-1}(M)$, one has the following so called Stokes Formula
\[\int_{W^u(x)}d\phi=\sum_{y\in \Crit(f), n_f(y)=n_f(x)-1}m(x,y)\int_{W^u(y)}\phi.\]
\end{prop}

For our noncompact case with tame conditions and Thom-Smale transversality, we have similarly 
\begin{prop}\label{stoke}
For any critical point $x\in \Crit(f)$ with Morse index $n_f(x)$, any $\phi\in \O^{n_f(x)-1}_c(M)$, one has the following so called Stokes Formula
\[\int_{W^u(x)}d\phi=\sum_{y\in \Crit(f), n_f(y)=n_f(x)-1}m(x,y)\int_{W^u(y)}\phi.\]
\end{prop}

Before giving the proof of this proposition, we first draw a couple of consequneces.
\begin{cor}
Let $\patt:C_*(W^u)\mapsto C_{*-1}(W^u)$ be the map constructed in Section \ref{tsw}, then $\patt^2=0$.
\end{cor}
\begin{proof}
Otherwise, $\patt^2 W^u(x)\neq0$. Then there exists $\phi\in \Omega_c^{n_f(x)-2}(M),$ s.t.
\[\int_{\patt^2W^u(x)}\phi\neq 0.\]
But by Proposition \ref{stoke},
\[\int_{\patt^2W^u(x)}\phi=\int_{W^u(x)}d^2\phi= 0\]
\end{proof}

\begin{cor}
Let $\o\in F^{[0,1],n_f(x)-1}_{Tf}$, one has
\[\int_{W^u(x)}\exp(Tf)d_{Tf}\o=\sum_{y\in \Crit(f), n_f(y)=n_f(x)-1}m(x,y)\int_{W^u(y)}\exp(Tf)\o.\]
\end{cor}
\begin{proof}
By Theorem \ref{est} and Lemma \ref{vol}, for any $\epsilon>0$, there exist $\phi\in \O^{n_f-1}_c(M)$, s.t. for any $y\in \Crit(f)$ with $n_f(y)=n_f(x)-1$,
\[\int_{W^u(x)}|\exp(Tf)d_{Tf}\o-d\phi|<\epsilon,\int_{W^u(y)}|\exp(Tf) \o-\phi|<\epsilon.\]

Now our Corollary follows from Proposition \ref{stoke} trivially. 
\end{proof}
Hence, the map $\J$ introduced in Section \ref{tsw} is a chain map.

The proof of Proposition \ref{stoke} follows from the following observations:
\begin{ob}\label{transversal}
Let $(N, \partial N)$ be compact manifold with boundary. Moreover, assume that near the boundary $\partial N$, the manifold is of  product type $(0,1]\times \partial N.$ Suppose that $f$ is a Morse function on $N-[1/2,1]\times \partial N$. Then there exist a transversal Morse function $\bar{f}$ on $N$, s.t. $\bar{f}|_{N-[1/4,1]\times \partial N}=f, \tilde{f}_{[3/4,1]\times\partial N}=r.$ Here $r$ is the standard coordinate on $(0,1]$ factor.
\end{ob}
The proof is essentially the same with Theorem 2.5 in \cite{milnor2015lectures}.

\begin{ob}\label{flow}
Let $\tilde{B}_D(p_0)$ be the ball with radius $D$ introduced in Section \ref{tsw}. Then for any $y,z\in \Crit(f)$, $W^u(y)\cap W^s(z)\subset \tilde{B}_D^\circ.$ Moreover, if $p\notin \tilde{B}_D(p_0)$ lies in an unstable manifold, then the curve $\{\Phi^t(p):t\geq0\}\cap\tilde{B}_D=\emptyset.$ Here $\tilde{B}_D^\circ$ denotes the interior of $\tilde{B}_D=\tilde{B}_D(p_0),$ $\Phi^t$ is the flow generated by $-\nabla f.$
\end{ob}
\begin{proof}
It follows easily from the fact that for any $p\in M$, $\td(p,\Phi^t(p))=|f(p)-f(\Phi^t(p))|/b$, and $f$ is decreasing along the flow $\Phi^t.$
\end{proof}
Now we are ready to prove Proposition \ref{stoke}
\begin{proof}
For any $\phi\in \O^{n_f(x)}_c(M)$, let \[D=bT(\sup_{p\in \Crit{f}\cup {\mathrm supp}(\phi)}|f(p)-f(p_0)|+\sup_{p,q\in \Crit{f}\cup {\mathrm supp}(\phi)}|f(p)-f(q)|),\]
we can find a compact submanifold $(N,\partial N)$ with boundary, s.t. ${\mathrm supp}(\phi)\cup \tilde{B}_D\subset N^\circ.$ Here ${\mathrm supp}(\phi)$ denotes the support of $\phi,$ $N^\circ$ denote the interior of $N.$

Now consider the double $(DN=N^+\cup N^-,g_{DN})$ of $N$,  $g_{DN}|_{{\mathrm supp}(\phi)\cup \tilde{B}_D}=g.$ By Observation \ref{transversal}, we can find a Morse function $\bar{f}$ on $DN$, s.t. $\bar{f}|_{{\mathrm supp}(\phi)\cup \tilde{B}_D}=f.$ We may as well assume that $(DN,g_{DN}, \tilde{f})$ satisfy Thom-Smale transversality condition. Then for any $y,z\in \Crit(\bar{f})$ with $n_{\bar{f}}(y)=n_{\bar{f}}(z)+1,$ let $m_{DN}(y,z)$ be the signed count of the  number of flow lines in $W^u_{\bar{f}}u(y)\cap W^s_{\bar{f}}(z).$

Then we claim:
\begin{enumerate}
\item \label{c1}By Observation \ref{flow}, if $y,z \in {\mathrm supp}(\o)\cup \tilde{B}_D$ are critical points of $\bar{f}$ with $n_{\bar{f}}(y)=n_{\bar{f}}(z)+1$, we have $m_{DN}(y,z)=m(y,z).$

\item \label{c2}If $z\notin {\mathrm supp}(\phi)\cup \tilde{B}_D(p_0),y\in {\mathrm supp}(\phi)\cap \tilde{B}_D$ are critical points of $\bar{f}$, and $W^s_{\bar{f}}(z)\cap W^u_{\bar{f}}(y)\neq \emptyset$, then $W^u_{\bar{f}}(z)\cap {\mathrm supp}(\phi)=\emptyset.$ This is because, by definition of $D$, Claim in Lemma \ref{rho} and properties of unstable manifolds, $\bar{f}(z)<\bar{f}(y)+|f(y)-f(p_0)|-D/(bT)\leq\inf_{p\in {\mathrm supp}(\phi)}\bar{f}(p).$ 
\end{enumerate}
As a result, by Proposition \ref{laudenbach}
\begin{align*}
\int_{W^u_f(x)}d\phi&=\int_{W^u_{\bar{f}}(x)}d\phi=\sum_{z\in \Crit(\bar{f}),n_{\bar{f}}(z)=n_{\bar{f}}(x)-1}m_{DN}(x,z)\int_{W^u_{\bar{f}}(z)}\phi\\
&=\sum_{y\in \Crit({f}),n_f(y)=n_f(x)-1}m_{DN}(x,y)\int_{W^u_{f}(y)}\phi\mbox{ (By Claim \ref{c2})}\\
&=\sum_{y\in \Crit({f}),n_f(y)=n_f(x)-1}m(x,y)\int_{W^u_f(y)}\phi\mbox{ (By Claim \ref{c1})}.
\end{align*}
\end{proof}
\subsubsection{An counterexample}
On closing this subsection, let's give a counterexample that when we drop the condition that $\nabla f$ has a positive lower bound near infinity,
$\tilde{\partial}^2=0$ may fail.

Consider the following heart shaped topological sphere $S$ with obvious height function $f$. Then we have four critical points $x,y,z,w$ as indicated below. Let $\gamma$ be a flow line connecting $y$ and $w$, and remove a point $p$ on $\gamma$. Making a conformal change of metric near point $p$, s.t. $S-p$ is complete under this new metric. Now we can see that $|\nabla f(q)|\to0$, as $q\to p.$ Since the flow line is invariant under the conformal change of metric, $\gamma-p$ is still a (broken) flow line. However, in this case, $\tilde{\partial}^2x=w$, which is nonzero.

In our previous arguments, the fact that $|\nabla f|$ has a positive lower bounded near the infinity play a crucial role. It forces the value of Morse function $f$ goes to infinity along a flow line if that flow line flows to infinity (See also Remark \ref{fgoestoinfinity} and Observation \ref{flow}). 

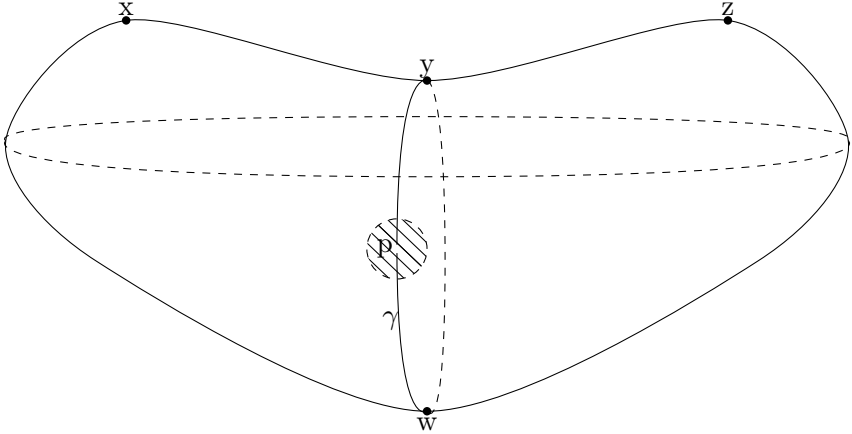
\begin{center}
\begin{tikzpicture}[scale=0.8]
\draw plot [smooth cycle] coordinates {(-5.5,4) (-7,6) (-5,8) (0,7) (5,8) (7,6) (5.5,4) (0,1.5)};
\draw [dashed] plot [smooth, tension=2] coordinates { (-7,6) (0,6.4) (7,6)};
\draw [dashed] plot [smooth, tension=2] coordinates { (-7,6) (0,5.4) (7,6)};
\draw plot [smooth, tension=2] coordinates { (0,7) (-0.5,4.2) (0,1.5)};
\draw [dashed] plot [smooth, tension=2] coordinates { (0,7) (0.3,4) (0,1.5)};
\draw [dashed](-0.5,4.2) circle (0.5cm);

\fill [white] (-0.5,4.2) circle (2pt);
\node (b) at (-0.7,4.2) {p};
\node (b) at (-0.6,3) {$\gamma$};
\fill [black] (0,7) circle (2pt);
\node (a) at (0,7.2) {y};
\fill [black] (-5,8) circle (2pt);
\node (b) at (-5,8.2) {x};
\fill [black] (5,8) circle (2pt);
\node (b) at (5,8.2) {z};
\fill [black] (0,1.5) circle (2pt);
\node (b) at (0,1.3) {w};
\begin{scope}
        \clip (-0.5,4.2) circle (0.5cm);
         \foreach \x in {-1,-0.75,-0.5,-0.25,0,0,0.25,0.5,0.75,1,1.25,1.5,1.75,2,2.25,2.5,2.75}
        \draw [xshift=\x cm](-2.6,6.3)--(1.6,2.3);
    \end{scope}
\end{tikzpicture}
\end{center}
\begin{rem}
    We would like to thank Shu Shen for providing this interesting example.
\end{rem}
\subsection{Isomorphism of $H^*(C^\bullet(W^u),\pat)$ and $H^*_{dR}(M,U_c)$}\label{relative}
For simplicity, we assume that $f$ is self-indexed Morse function, i.e., if $x$ is a critical point of $f$ with Morse index $i$, we require $f(x)=i$. Moreover, $B_1(x)$

Let $V_i=f^{-1}(-\infty,i+\frac{1}{2}]$, $0\leq i\leq n$.

Recall that we assume in a neighborhood $U_x$ of critical points $x$ of $f$, we have coordinate system $z=(z_1,...,z_n),$ such that
\[f=-z_1^2-...-z_{n_f(x)}^2+z_{n_f(x)+1}^2+...+z_n^2,\]
\[g=dz_1^2+...+dz_n^2,\]
Moreover $U_x$ is an Euclidean open ball around $x$ with radius $1.$ Also, these open balls are disjoint.

We have the following observation:
\begin{lem}\label{v0vn}
$V_0$ can be written as disjoint union of $\cup_{x\in Crit(f),n_f(x)=0}\tilde{U}_x$ and  $V$, where $V$ is some open subset diffeomorphic to $U_c$, $\tilde{U}_x$ is an Euclidean ball around $x$ with radius $\frac{1}{2}$. 

$V_n$ is diffeomorphic to $M.$
\end{lem}
\begin{proof}
Let $X_f:=\frac{\nabla f}{|\nabla f|^2}$, $\Phi^t$ be the flow generated by $X_f$. Then we have \[\left(\Phi^{c+\frac{1}{2}}(U_c)\right)\cap\left(\cup_{x\in Crit(f),n_f(x)=0}\tilde{U}_x\right)=\emptyset.\]

This is because:
\begin{itemize}
    \item If $f(p)\leq c-\frac{1}{2}$, then $f(\Phi^{c+\frac{1}{2}}(p))<0$. Hence $\Phi^{c+\frac{1}{2}}(p)\notin \cup_{x\in Crit(f),n_f(x)=0}\tilde{U}_x$.
    \item If $c-\frac{1}{2}\leq f(p)<c$, and if $\Phi^{c+\frac{1}{2}}(p)\in \tilde{U}_x$ for some $x\in Crit(f)$ with Morse index $n_f(x)=0$. Then $\Phi^{c+\frac{1}{2}}(p)\in W^s(x)$, which implies $p\in W^s(x)$. But this is impossible since $f(p)<-c<0=f(x)$. 
\end{itemize}

Similarly, we can prove that $V_n$ is diffeomorphic to $M.$
\end{proof}

Let $C_*(V_i,U_c)$ be complex of relative singular chains. Then we have
\[C_*(V_n,U_c)\supset C_*(V_{n-1},U_c)\supset\cdots C_*(V_{0},U_c).\]

By a similar spectral sequence argument as in the proof of Theorem 1.6 in \cite{bismutzhang1992cm} and Lemma \ref{v0vn}, one can show that 
\[H_*(C^\bullet(W^u),\pat)\simeq H_*(M,U_c).\]
Thus, it follows from the universal coefficient theorem that
\[H^*(C^\bullet(W^u),\pat)\simeq H^*(M,U_c).\]

\subsection{Isomorphism of $H^*_{(2)}(M,d_{Tf})$ and $H^*(C^\bullet(W^u),\pat)$ }\label{Independent}

We will first show that the chain map $\J: (F_{Tf}^{[0,1],*},d_{Tf})\mapsto C^*((W^u)',\patt')$ defined in Section \ref{tsw} is in fact an isomorphism when $T$ is sufficiently large. Hence $\mathcal{J}$ induced an isomorphism between $H^*_{(2)}(M,d_{Tf})$ and $H^*(C^\bullet(W^u),\pat)$ in that case. 

More precisely we will follow the arguments in Chapter 6 of \cite{zhang2001lectures}, with necessary modification, to show that there exists $T_6>T_0$, such that $\J$ is an isomorphism whenever $T>T_6.$  (We point out that the explicit description of $T_6$ is more involved than $T_0$.)
In fact, the only difference is that we need a more refined estimate in Theorem 6.17 of \cite{zhang2001lectures}, that is:
\begin{equation}\label{thm617}|\mathcal{P}\tau_{x,T}-\tau_{x,T}|\leq C\exp(-aT(\rho+c)) \|\tau_{x,T}\|_{L^2},\end{equation}
where $\mathcal{P}$ is the orthogonal projection from $L^2\Lambda(M)$ to $F^{[0,1],*}$, and C, c are positive constants.

Here $\tau_{x,T}$ is defined as follows. Notice that in Section \ref{MorseInequality}, we require that in a neighborhoof $U$ of $x$, the metric and Morse function is of the form (\ref{flatMetric}).
Hence, let $\alpha_{x}$ be a bump function whose support is contained in $U$, and $\alpha_{x}\equiv1$ in a neighborhood $V$ of $x$. Now let 
\[\tau_{x,T}=\alpha_{x}\exp(-T^2|z|^2)dz_1\wedge \cdots\wedge  dz_{n_f(x)}.\]

To obtain the estimate (\ref{thm617}), pick a bump function $\eta$ with compact support, such that $\eta\equiv1$ on $K$. Then by our Agmon estimate, we have \[|(1-\eta)(\mathcal{P}\tau_{x,T}-\tau_{x,T})|\leq C\exp(-aT(\rho+c))\|\tau_{x,T}\|_{L^2}.\]
The estimate of \[|\eta(\mathcal{P}\tau_{x,T}-\tau_{x,T})|\leq C\exp(-cT)\|\tau_{x,T}\|_{L^2}\] 
now follows from exactly the same argument in the proof of Theorem 6.17 of \cite{zhang2001lectures}.

Now it remains to prove that when $T\in(T_0, T_6]$, $H^*_{(2)}(M,d_{Tf})$ and $H^*(C^\bullet(W^u),\pat)$ are still isomorphic.

We only present the proof for the case when $(M,g,f)$ is strongly tame, the case of well tame being exactly the same except notationally. In this case, $T_0=0$. The idea is to show that if $S>0,$ then for any $T\in [7/8S,S]$, $H^*_{(2)}(M,d_{Tf})$ and $H^*_{(2)}(M,d_{Sf})$ are isomorphic. Hence $H^*_{(2)}(M,d_{Tf})$ is independent of $T\in (0,\infty)$, which finishes the proof of isomorphism of $H^*_{(2)}(M,d_{Tf})$ and $H^*(C^\bullet(W^u),\pat)$.

For simplicity, we prove that $H^*_{(2)}(M,d_{7f})$ and $H^*_{(2)}(M,d_{8f})$ are isomorphic, the general case being similar.

Thus fix coefficient $b=\frac{63}{64}$ in Lemma \ref{t0} and Theorem \ref{est}.

\def\ls{L^2\Lambda^*(M)}

Define $M_{f}:(F^{*,[0,1]}_{8f},d_{8f}) \mapsto (\Omega^*_{(2)}(M), d_{7f})$;  $\forall w\in F^{*,[0,1]}_{8f},$ $M_f(w)=\exp(f)w$. Similarly
$M_{-f}: (F^{*,[0,1]}_{7f}, d_{7f}) \mapsto (\Omega^*_{(2)}(M), d_{8f})$; $\forall w\in F^{*,[0,1]}_{7f},$ $M_{-f}(w)=\exp(-f)w.$

Clearly these are chain maps once we check that $M_f$ and $M_{-f}$ are well defined.
To this end, let's verify that $|f(p)|\leq |f(p_0)|+\frac{1}{bT}\rho_T(p),$ where $p_0$ is the fix point in defining $\rho_T.$ Indeed, let $\gamma:[0,\rho_{T}(p)]$ be a normal minimal geodesic connecting $p_0$ and $p$, in the metric $\gt$. Then
\[|\frac{d}{dt}f\circ\gamma(t)|=|<\tilde{\nabla}f,\gamma'>_{\gt}|\leq\frac{1}{bT}.\]

Now the $L^2$ bound of $M_f(w)$ (resp. $M_{-f}(w)$) follows by Theorem \ref{est} and the standard volume comparison, 
Hence $M_f$ induces a homomorphism (still denote it by $M_f$) from $H^*_{(2)}(M,d_{8f})$ to $H^*_{(2)}(M,d_{7f})$.

Our next step is to show that $M_f$ is injective. Suppose we have $w\in \ker(\Box_{8f})$, s.t. $M_fw$ is exact, which means that we can find $\alpha\in Im(\delta_{7f}),$ s.t $\exp(f)w=d_{7f}\alpha(=(d_{7f}+\delta_{7f})\alpha).$

Thus
\[\Box_{7f}\alpha=(d_{7f}+\delta_{7f})\exp(f)w=\exp(f)d_{8f}w+\exp(2f)\delta_{6f}w\]
\[=0+\exp(2f)(\delta_{8f}w-\iota_{2f}w)=-\exp(2f)\iota_{2f}w.\]

By Lemma \ref{epde}, $|\alpha|\leq C\exp(-1/3\rho_7).$ Consequently, $\exp(-f)\alpha\in \ls$, and $w=d_{8f}\exp(-f)\alpha$ is exact.

As a result, $M_f$ is injective.
Similarly, $M_{-f}$ is also injective. Therefore, $H^*_{(2)}(M,d_{8f})$ and $H^*_{(2)}(M,d_{7f})$ are isomorphic.

\subsection{$\L$ is bijective}\label{bijec}
\def\bPhi{\bar\Phi}
\begin{itemize}
\item $\L$ is injective:

Let $\o\in F^{[0,1],j}_{Tf}$. Assume that $\L(\o)$ is exact, then there exists $\phi\in\Omega^{j-1}(M),\phi'\in \Omega^{j-2}(U_c),$ s.t.
\[\exp(Tf)\o=d\phi, \o'=(-1)^j(\phi-d\phi'),\]
where $\o'=(-1)^j\int_0^\infty(\bPhi^s)^*(\exp(Tf)\iota_{X_f}\o)ds.$

Let $c'=c+2$, then choose a smooth function $\chi : M \mapsto \R$ such that
$\chi|_{U_{c'}} = 1$ and $\chi|_{M-U_c}=0$
Then denoting $\psi= \phi - d(\chi\phi')$, we have
\[d\psi = \exp(Tf)\o \mbox{ and } \psi|_{U_{c'}} = (-1)^j\o' .\]

Also, on $U_{c'},$ \begin{equation}\label{mu} \iota_{X_f}\psi=\iota_{X_f}\int_0^\infty(\bPhi^s)^*(\exp(Tf)\iota_{X_f}\o)ds=\int_0^\infty(\bPhi^s)^*(\exp(Tf)\iota_{X_f}\iota_{X_f}\o)ds=0\end{equation}

Next, choose a smooth function $\eta,$ s.t. $\eta=0$ on $K$, and $\eta=1$ in $f^{-1}((S+1,\infty)\cup(-\infty,-I-1))$.
Then we set
\[\psi'(p) = \psi(p) -d(\eta\int_{-f(p)+I/2+S/2}^0(\bPhi^s)^*\iota_{X_f}\psi ds), p\in U_c',\]
\[\psi'(p) = \psi(p) -d(\eta\int^{-f(p)+I/2+S/2}_0(\bPhi^s)^*\iota_{X_f}\psi ds), p\in U_c,\]
By (\ref{mu}), we have $\psi'=\psi$ on $U_{c'}$.

Thus we can see that for $p\in U_{c'}'$,
\begin{align*}
\iota_{X_f}\psi'&=\iota_{X_f}\psi-\iota_{X_f}d(\eta\int_{-f(p)+I/2+S/2}^0(\bPhi^s)^*\iota_{X_f}\psi ds)\\
&=\iota_{X_f}\psi-(\int_{-f(p)+I/2+S/2}^0\iota_{X_f}d(\bPhi^s)^*\iota_{X_f}\psi ds)\\&-\iota_{X_f}df(\bPhi^{-f(p)+I/2+S/2})^*\iota_{X_f}\psi\\
&=\iota_{X_f}\psi-(\int_{-f(p)+I/2+S/2}^0\frac{d}{ds}(\bPhi^s)^*\iota_{X_f}\psi ds-(\bPhi^{-f(p)+I/2+S/2})^*\iota_{X_f}\psi\\
&=\iota_{X_f}\psi-\iota_{X_f}\eta\psi+(\bPhi^{-f(p)+I/2+S/2})^*\iota_{X_f}\psi-(\bPhi^{-f(p)+I/2+S/2})^*\iota_{X_f}\psi\\
&=0.
\end{align*}
Similarly, we have $\iota_{X_f}\psi'=0$ in $U_{c'}.$

As a consequence,
\[\frac{d}{ds}(\bPhi^s)^*\psi'|_{s=t}=\iota_{X_f}(\bPhi^t)^*d\psi'=\iota_{X_f}(\bPhi^t)^*\exp(Tf)\o\]
on $U.$

Therefore, on $U_c'$, we have
\begin{align}\begin{split}\label{ucc}|(\bPhi^t)^*\psi'|&=|\int_0^t\iota_{X_f}(\bPhi^s)^*\exp(Tf)\o ds|\\
&\leq \int_0^t|\iota_{X_f}(\bPhi^s)^*\exp(Tf)\o| ds\\
&\leq \int_0^t|(\bPhi^s)^*\exp(Tf)\o| ds\\
&\leq C\exp((1-a)Tt)\mbox{ (Since on $U_c'$, $(\bPhi^t)^*\o\leq C\exp(-aTt))$.}
\end{split}
\end{align}

We claim that $\exp(-Tf)\psi'\in L^2\Lambda^*(M).$ Since $d\psi'=\exp(Tf)\o$ implies that $d_{Tf}\exp(-Tf)\psi'=\o$,  $\o$ is trivial in $H^*(\Omega_{(2)}^{\bullet},d_{Tf}).$

Now we prove the claim:

It suffices to prove that $\int_{U_{c'}\cup U_c'}|\exp(-Tf)\psi'|^2dvol<\infty.$ Let $K_{c'}=f^{-1}\{-c'\}, K_c'=f^{-1}\{c\},$ and endow they with induced metrics. Define a diffeomorphism
$\Psi_{c'}:K_{c'}\times(0,\infty)\mapsto U_{c'}$ as follow
\[\Psi_{c'}(p,t)=\bPhi^t(p).\]
Similarly, we can define a diffeomorphism $\Psi_{c}':K_{c}'\times(0,\infty)\mapsto U_{c}'.$

On $U_{c'}$, $|\psi'|=|\o'|$, hence for $p\in K_{c'}$
\begin{align}\begin{split}\label{equc}
&|(\bPhi^t)^*(\exp(-Tf)\psi')(p)|=(\bPhi^t)^*(\exp(-Tf)(p)\int_0^\infty(\bPhi^s)^*(\exp(Tf)\iota_{X_f}\\o(p))ds)\\
&=\exp(Tc'+Tt)\int_0^\infty((\exp(-Tc'-T(s+t))(\bPhi^{s+t})^*\iota_{X_f}\o(p))ds)\\
&\leq \int_0^\infty((\exp(-Ts)|(\bPhi^{s+t})^*\iota_{X_f}\o(p))|ds\\
&\leq C\exp(-aTt)\int_0^\infty((\exp(-(a+1)Ts)ds\mbox{ (Since $(\bPhi^{s+t})^*\o\leq C\exp(-T(s+t)))$}\\
&\leq C\exp(-aTt).
\end{split}\end{align}
Then,
\begin{align*}
&|\int_{U_{c'}}|\exp(-Tf)\psi'|^2dvol|=|\int_0^\infty \int_{K_{c'}}(\bPhi^t)^*(|\exp(-Tf)\o'|^2dvol_{K_{c'}}dt)|\\
&\leq\int_0^\infty \int_{K_{c'}}\exp(-2aTt)(\bPhi^t)^*(dvol_{K_{c'}}dt)\\
&\leq C\int_0^\infty \int_{K_{c'}}\exp(-2a'Tt)dvol_{K_{c'}}dt<\infty, \mbox{ (Similar with Lemma \ref{vol})}
\end{align*}
where $a'$ is some positive number which is smaller than $a.$

For $p\in U_c',$ we have
\begin{align*}
&|(\bPhi^t)^*(\exp(-Tf)\psi')(p)|=|(\exp(-Tc-Tt)(\bPhi^t)^*\psi')(p)|\\
&\leq C\exp(-aTt) \mbox{ (By (\ref{ucc}))}.
\end{align*}
For the same reason, we have $\int_{U_{c}'}|\exp(-Tf)\psi'|^2dvol<\infty$.
\item $\L$ is surjective:
\def\o{\omega}
We claim that any cohomology class $\xi\in H^j(M, U_{c})$ can be
represented by a smooth closed $j$-form $\phi$ so that $\phi|_{U_c}=0.$
Also, it behaves on $U$ as follows: $\iota_{X_f}\phi = 0$ and $(\bPhi^t)^*\phi$ does not depend on $t$ for large $t$. Then it
follows that $\exp(-Tf)\phi$ belongs to $L^2\Lambda^*(M)$ (follows from the similar argument as above). Let $\nu\in Ker\Box_{Tf},$ s.t. $\nu-\exp(-Tf)\phi$ is exact. Then we can see that $\L(\nu)\in\xi$ hence $\L$ is surjective. This is because, we can find $\psi\in Im(\delta_{Tf})$, s.t. $\nu-\exp(-Tf)\phi=d_{Tf}\psi=(d_{Tf}+\delta_{Tf})\psi.$ As a result, $\Box_{Tf}\psi=0$ on $U_c$. Hence $\psi$ is of exponential decay in $U_c$, which implies that $\L(\psi)$ is well define. Now $\L(\nu)-\L(\exp(-Tf)\phi)=d_C\L(\psi)$ implies that $\L(\nu)\in \xi.$

It suffices to prove the claim:

It is clear that we may realize $\xi$ by a closed form $\phi$ on $M$ with $d\phi = 0$ and
$\phi|_{U_{c}}=0$.
Let $\eta: M \mapsto \R$ denote a smooth function which is identically $0$ on $K$ and identically $1$ on $f^{-1}((S+1,\infty)\cup(-\infty,-I-1))$. The form
$\phi'(p) = \phi(p) -d(\eta\int_{-f(p)+I/2+S/2}^0(\bPhi^s)^*\iota_{X_f}\phi ds)$
is cohomologous to $\phi$ and clearly satisfies what we claimed.
\end{itemize}

\section{Appendix: Decomposition of $L^2$ space}\label{l2}
In this section, we investigate the decomposition (\ref{dec}). For this purpose we first have to understand the Friedrichs extension of $\Delta_{H,f}$. Moreover, all operators considered in this section are closurable.
\subsection{Review on Friedrichs extension}
\def\B{\operatorname{B}}
\def\A{\operatorname{A}}
\def\I{\operatorname{I}}
Let $\A$ be a nonnegative, symmetric (unbounded) operator on Hilbert space $\H$, with $\Dom(\A)=V,$ i.e.
\[(\A\alpha,\beta)_\H=(\alpha,\A\beta)_\H, \ \forall \alpha,\beta\in V; \ \ (\A\alpha,\alpha)_\H\geq 0.\]

Define a norm $\|\cdot\|_{V_1}$ on $V$ by
\[\|\alpha\|_{V_1}^2=(\alpha,\alpha)_\H+(\alpha,\A\alpha)_\H.\]
Let $V_1$ to be the completion of $V$ under $\|\cdot\|_{V_1}.$ Then for any $\beta\in \H$, one can construct a bounded linear functional $L_{\beta}$ on $V_1$ as follows
\begin{equation}\label{eq}
	L_{\beta}(\phi)=(\phi,\beta)_\H,\phi\in V_1.
\end{equation}
Since $|(\phi,\beta)_\H|\leq \|\phi\|_\H\|\beta\|_\H\leq \|\phi\|_\H\|\beta\|_{V_1},$ $L_\beta$ is indeed bounded functional on $V_1.$ By Riesz representation, there exist $\gamma\in V_1$, s.t. $(\phi,\gamma)_{V_1}=(\phi,\beta)_\H.$

Let $\B:\H\mapsto V_1, \beta\mapsto \gamma,$ then $\B$ is bounded and injective. Take $\Box=\B^{-1}-\I,$ where $\I$ is the identity map, then $\Box$ is the Friedrichs extension of $\A,$ with $\Dom(\Box)=\Im(\B).$

\begin{rem}
	From the construction of Friedrichs extension $\Box$ of $\A$, we can see that $Dom(\Box)=\Im((\I+\Box)^{-1}).$
\end{rem}
\def\S{\operatorname{S}}
Let $\T,\S$ be two unbounded operators on Hilbert space $\H$, s.t.

\begin{enumerate}[(i)]
	\item\label{i1} \[V=\Dom(\T)=\Dom(\S), \T V\subset V.\]
	\item \label{i2} $S$ is a formal adjoint of $T:$  \[(\T\alpha,\beta)_\H=(\alpha,\S \beta)_\H,\]
\end{enumerate}

Let $\|\cdot\|_{W}$ be the norm on $V$ given by
\[\|\alpha\|^2_{W}=(\alpha,\alpha)_\H+(\T\alpha,\T\alpha)_\H,\alpha\in V,\]
and $W$ be the completion of $V$ under the norm $\|\cdot\|_{W}.$ Then we can extend $\T$ to $\bar{\T}_{min}$ with $\Dom(\bar{\T}_{min})=W.$

Let $\bar{\S}_{max}$ be the closure of $\S$ with $\Dom(\bar{\S}_{max})=\{\alpha\in\H:|(\alpha,\T\phi)_\H|\leq M_\alpha\|\phi\|_\H,\forall\phi\in V\}$. Namely, for any $\alpha\in \Dom(\bar{\S}_{max})$, since $V$ is dense in $\H$, by Riesz representation, there exists unique $\nu\in \H$, such that $(\nu,\phi)_H=(\alpha,\T\phi). $ Now define $\bar{\S}_{max}(\alpha)=\nu.$

Since $\T V\subset V$, $\S\T$ is symmetric and nonnegative with $\Dom(ST)=V.$

\begin{prop}\label{prop}
	The Friedrichs extension $\Delta$ of $\S\T$ is just $\bar{\S}_{max}\bar{\T}_{min}.$
\end{prop}

\begin{proof}
	Since $\T V\subset V,$ we see that $V_1$ constructed in (\ref{eq}) is the same as $W.$
	Indeed, for any $\phi,\psi\in V,$ we have
	\begin{align*}
		(\psi,\phi)_\H+(\T\psi,\T\phi)_\H
		=(\psi,\phi)_\H+(\S\T\psi,\phi)_\H
	\end{align*}
	Hence, we have
	\[\Dom(\Delta)=\{\alpha\in W: \alpha=(I+\Delta)^{-1}f, f\in\H\},\]
	\[\Dom(\bar{S}_{max}\bar{T}_{min})=\{\alpha\in W: \bar{\T}_{min}\alpha\in \Dom(\bar{\S}_{max})\}.\]
	We now divide our discussion in two cases.
	
	(a) We first prove that $\Dom{\bar{\S}_{max}\bar{\T}_{min}}\subset \Dom(\Delta),$ and $\forall\alpha\in \Dom(\bar{\S}_{max}),$ $\bar{\S}_{max}\bar{\T}_{min}\alpha=\Delta\alpha.$
	
	For any $\alpha\in\Dom{\bar{\S}_{max}\bar{\T}_{min}},$ let \begin{equation}\label{eq2}\beta=\alpha+\bar{\S}_{max}\bar{\T}_{min}\alpha.\end{equation} Then for any $\phi\in W,$ we have
	\begin{align}
		\begin{split}\label{eq1}
			&(\alpha,\phi)_{W}=\lim_{n\rightarrow\infty}(\alpha,\phi_n)_{W}\\
			&=\lim_{n\rightarrow\infty}(\alpha,\phi_n)_\H+(\bar{\T}_{min}\alpha,\T\phi_n)_\H\\
			&=\lim_{n\rightarrow\infty}(\alpha,\phi_n)_\H+(\bar{\S}_{max}\bar{\T}_{min}\alpha,\phi_n)_\H\mbox{ (Since $\phi_n\in V,\bar{\T}_{min}\alpha\in \Dom({\bar{\S}_{max}})$ )}\\
			&=\lim_{n\rightarrow\infty}(\alpha+\bar{\S}_{max}\bar{\T}_{min}\alpha,\phi_n)_\H=(\alpha+\bar{\S}_{max}\bar{\T}_{min}\alpha,\phi)_\H=(\beta,\phi)_\H,
		\end{split}
	\end{align}
	where $\phi_n\in V$, and $\phi_n\rightarrow \phi$ w.r.t. $\|\cdot\|_W.$
	By the construction of Friedrichs extension and (\ref{eq1}), we deduce that  $\alpha\in (I+\Delta)^{-1}\H$ and $(I+\Delta)\alpha=\beta.$ Comparing with (\ref{eq2}), we obtain $\bar{\S}_{max}\bar{\T}_{min} \alpha=\Delta \alpha.$
	
	(b) We then show that $\Dom(\Delta)\subset \Dom(\bar{\S}_{max}\bar{\T}_{min}).$
	
	Take any $\alpha\in \Dom(\Delta)\subset W,$ we can find $f\in \H$, s.t. $\alpha=(I+\Delta)^{-1}f.$ We now just need to show that $\bar{\T}_{min}\alpha\in \Dom(\bar{\S}_{max}).$ For this, it suffices to prove that $\forall g\in V,$ $|(\bar{\T}_{min}\alpha,\T g)_\H|\leq M\|g\|_\H$ for some $M>0.$
	
	In fact, by standard functional calculus,
	\begin{align*}
		&|(\bar{\T}_{min}\alpha,\T g)_\H|=|(\alpha, STg)_\H| \mbox{ (via $\alpha_n\in V, \alpha_n\rightarrow \alpha $ w.r.t $\|\|_{W}$)}\\
		&=|((I+\Delta)^{-1}f,\Delta g)_\H|\\
		&=|(f,(I+\Delta)^{-1}\Delta g)_\H| \\
		&\leq M\| g\|_\H
	\end{align*}
\end{proof}

\subsection{The Friedrichs extension of $\Delta_{H,f}$}

By Proposition \ref{prop}, we can see that the Friedichs extension $\Box_{f}$ of $\Delta_{H,f}$ is $(\overline{d_f+\delta_f})_{max}(\overline{d_f+\delta_f})_{min}$.

If $0$ is an eigenvalue of $\Box_{f}$ with finite multiplicity, we have the following decomposition
\begin{equation}\label{eq3}L^2\Lambda^*(M)=\ker\Box_{f}\oplus\Im(\overline{d_f+\delta_f})_{max}.\end{equation}

Could we say more about decomposition (\ref{eq3})?
\begin{prop}\label{prop1}
	\def\S{\operatorname{S}}
	Let $\T,\S$ be two unbounded operators on Hilbert space $\H$, such that
	\begin{enumerate}
		\item \[V=\Dom(\T)=\Dom(\S), \T V\subset V.\]
		\item $\Im(T)$ is orthogonal to $\Im(S)$, and \[(\T\alpha,\beta)_\H=(\alpha,\S \beta)_\H.\]
		\item $T+S$ is essential self-adjoint, i.e. $\overline{(T+S)}_{min}=\overline{(T+S)}_{max}.$
	\end{enumerate}
	Then 
	\[\overline{T+S}=\bar{T}_{min}|_{\Dom{\bar{S}_{min}\cap\Dom\bar{T}_{min}}}+\bar{S}_{min}|_{\Dom{\bar{S}_{min}\cap\Dom\bar{T}_{min}}}\]
	\[=\bar{T}_{max}|_{\Dom{\bar{S}_{max}\cap\Dom\bar{T}_{max}}}+\bar{S}_{max}|_{\Dom{\bar{S}_{max}\cap\Dom\bar{T}_{max}}}\]
	
\end{prop}

\begin{proof}
	Since $\Dom{\overline{(T+S)}}_{min}$ is the closure of $V$ under metric
	\[ (\phi,\phi)_\H+((T+S)\phi,(T+S)\phi)_\H=(\phi,\phi)_\H+(T\phi,T\phi)_\H+(S\phi,S\phi)_\H,\phi\in V, (**)\]
	Hence, $\Dom{\overline{(T+S)}}_{min}\subset\Dom{\bar{S}_{min}\cap\Dom\bar{T}_{min}}.$ Also, for any $\phi\in \Dom{(T+S)}_{min}$
	\[\overline{(T+S)}_{min}\phi=\lim_{n\rightarrow\infty} (T+S)\phi_n=\lim_{n\rightarrow\infty} T\phi_n+S\phi_n=\T_{min}\phi+\S_{min}\phi,\]
	where $\phi_n\in V\rightarrow\phi$ in the metric $(**)$.
	
	For each $\phi\in \Dom{\bar{S}_{max}\cap\Dom\bar{T}_{max}},$ $\psi\in V$,
	\begin{align*}
		&(\phi, (T+S)\psi)_\H=(\phi, T\psi)_\H+(\phi, S\psi)_\H\\
		&=(\bar\T_{max}\phi,\psi)_\H+(\bar\S_{max}\phi,\psi)_\H\\
		&\leq C\|\psi\|_{\H}.
	\end{align*}

	Therefore $\phi\in \Dom(\overline{(T+S)}_{max}),$ and $\overline{(T+S)}_{max}\phi=\bar\T_{max}\phi+\bar\S_{max}\phi,$ which means that $\Dom{\bar{S}_{min}\cap\Dom\bar{T}_{min}}\subset \Dom(\overline{(T+S)}_{max}).$
\end{proof}

\bibliography{lib}
\bibliographystyle{plain}
\end{document}